\documentclass{article}
\usepackage{amssymb}
\usepackage{amsmath}

\newtheorem{theorem}{Theorem}

\newtheorem{corollary}[theorem]{Corollary}

\newtheorem{definition}[theorem]{Definition}
\newtheorem{example}[theorem]{Example}

\newtheorem{proposition}[theorem]{Proposition}
\newtheorem{remark}[theorem]{Remark}

\newenvironment{proof}[1][Proof]{\noindent\textbf{#1.} }{\ \rule{0.5em}{0.5em}}
\input{tcilatex}

\begin{document}

\title{Boundary asymptotic analysis for an incompressible viscous flow:
Navier wall laws}
\author{M.\ El Jarroudi, A.\ Brillard \\
Universit\'{e} Abdelmalek Essa\^{a}di, FST Tanger\\
D\'{e}partement de Math\'{e}matiques, B.P. 416, Tanger, Morocco\\
Universit\'{e} de Haute-Alsace\\
Laboratoire de Gestion des Risques et Environnement\\
25 rue de Chemnitz, F-68200 Mulhouse, France}
\date{2008}
\maketitle

\begin{abstract}
We consider a new way of establishing Navier wall laws. Considering a
bounded domain $\Omega $ of $\mathbf{R}^{N}$, $N=2,3$, surrounded by a thin
layer $\Sigma _{\varepsilon }$, along a part $\Gamma _{2}$ of its boundary $%
\partial \Omega $, we consider a Navier-Stokes flow in $\Omega \cup \partial
\Omega \cup \Sigma _{\varepsilon }$ with Reynolds' number of order $%
1/\varepsilon $ in $\Sigma _{\varepsilon }$. Using $\Gamma $-convergence
arguments, we describe the asymptotic behaviour of the solution of this
problem and get a general Navier law involving a matrix of Borel measures
having the same support contained in the interface $\Gamma _{2}$. We then
consider two special cases where we characterize this matrix of measures. As
a further application, we consider an optimal control problem within this
context.

Navier law, Navier-Stokes flow, $\Gamma $-convergence, asymptotic behaviour,
optimal control problem.

\textbf{AMS\ Classification.}\ 76D05, 76D10, 76M45, 35Q30.
\end{abstract}

\section{Introduction}

A common hypothesis used in fluid mechanics is that, at the interface
between a solid and a fluid, the velocity $u$ of the fluid is equal to that
of the solid. If the solid is at rest, the velocity of the fluid must thus
vanish: $u=0$, on the boundary of the solid. These are the so-called rigid
boundary conditions. When writing this condition, one assumes that the fluid
perfectly adheres to the solid.

This hypothesis has not always been accepted for a viscous fluid, although
some verifications have been made through experiments.\ G.\ Taylor indeed
verified in 1923 the correctness of this hypothesis, when studying the
stability of the motion of a fluid flowing between two cylinders in rotation
(Taylor-Couette's problem).

Another approach has then been suggested. A thin layer adhering to the solid
exists with a tangential velocity different from 0 on the surface of the
solid. Navier suggested that this tangential velocity is proportional to the
shearing strains and thus is given through%
\[
\left\{ 
\begin{array}{rll}
\left( Id-n\otimes n\right) \nu \dfrac{\partial u}{\partial n} & = & \kappa
u, \\ 
u\cdot n & = & 0,%
\end{array}%
\right.
\]%
where $Id$ is the identity matrix, $n$ is the unit outer normal vector to
the surface of the solid, $\nu $ is the viscosity of the fluid and $\kappa $
is a proportionality coefficient.

Many works have already been devoted to the derivation of Navier boundary
conditions, see for example \cite{Ach-Pir}, \cite{Ach-Pir-Val}, \cite%
{Jag-Mik} and \cite{Mar}. In \cite{Ach-Pir} and \cite{Ach-Pir-Val}, the
authors considered a viscous and incompressible fluid, whose Reynolds number
is of order $1/\varepsilon $, flowing in a domain with rugosities of
thinness $\varepsilon $ and $\varepsilon $-periodically distributed on its
boundary surface, and assuming an homogeneous Dirichlet boundary condition
on the boundary of these rugosities. Using the asymptotic expansion method,
they deduced, at the first-order level, a kind of Navier wall law%
\[
\left\{ 
\begin{array}{rll}
\varepsilon \left( Id-n\otimes n\right) \nu \dfrac{\partial u}{\partial n} & 
= & \kappa u, \\ 
u\cdot n & = & 0.%
\end{array}%
\right.
\]

In \cite{Jag-Mik}, the authors considered the laminar flow in a pipe with
rough pieces $\varepsilon $-periodically distributed on the surface of the
pipe, and imposing an homogeneous Dirichlet boundary condition on the
boundary of these rough pieces. They used an homogenization process and
obtained a Navier wall law, computing a corrector term. In \cite{Mar}, the
author considered an $\varepsilon $-periodic geometry built with rough
pieces of thinness $\varepsilon ^{m}$ and imposed there a boundary condition
of the type%
\[
\left\{ 
\begin{array}{rll}
\left( Id-n\otimes n\right) \nu \dfrac{\partial u^{\varepsilon }}{\partial n}
& = & \varepsilon ^{k}\left( g^{\varepsilon }-\kappa u^{\varepsilon }\right)
, \\ 
u^{\varepsilon }\cdot n & = & 0.%
\end{array}%
\right.
\]

The following limit law was obtained, depending on $k$ and $m$%
\[
\left\{ 
\begin{array}{rll}
\left( Id-n\otimes n\right) \nu \dfrac{\partial u}{\partial n} & = & \lambda
\left( g-\kappa u\right) , \\ 
u\cdot n & = & 0.%
\end{array}%
\right.
\]

Throughout the present work, we consider a bounded domain $\Omega \subset 
\mathbf{R}^{N}$, $N=2,3$, whose boundary $\partial \Omega $ is Lipschitz
continuous. We suppose that $\partial \Omega =\Gamma _{1}\cup \Gamma _{2}$,
with $\left\vert \Gamma _{1}\right\vert $, $\left\vert \Gamma
_{2}\right\vert >0$, where $\left\vert \Gamma _{i}\right\vert $ denotes the
Lebesgue measure of $\Gamma _{i}$. We suppose that near $\Gamma _{2}$ there
exists a thin layer $\Sigma _{\varepsilon }$ of thinness $\varepsilon >0$,
which extends $\Omega $ into $\Omega _{\varepsilon }=\Omega \cup \Gamma
_{2}\cup \Sigma _{\varepsilon }$.

\FRAME{fhFU}{5.4542cm}{3.2202cm}{0pt}{\Qcb{The domain under consideration.}}{%
}{navierirr.bmp}{\special{language "Scientific Word";type
"GRAPHIC";maintain-aspect-ratio TRUE;display "USEDEF";valid_file "F";width
5.4542cm;height 3.2202cm;depth 0pt;original-width 1.7668in;original-height
1.0326in;cropleft "0";croptop "1";cropright "1";cropbottom "0";filename
'navierirr.bmp';file-properties "XNPEU";}}

We consider the steady-state, viscous and incompressible Navier-Stokes flow
in $\Omega _{\varepsilon }$%
\begin{equation}
\left\{ 
\begin{array}{rrll}
-\nu \Delta u^{\varepsilon }+\left( u^{\varepsilon }\cdot \nabla \right)
u^{\varepsilon }+\nabla p^{\varepsilon } & = & f & \text{in }\Omega , \\ 
-\nu \varepsilon \Delta u^{\varepsilon }+\left( u^{\varepsilon }\cdot \nabla
\right) u^{\varepsilon }+\nabla p^{\varepsilon } & = & f & \text{in }\Sigma
_{\varepsilon }, \\ 
\limfunc{div}\left( u^{\varepsilon }\right) & = & 0 & \text{in }\Omega
_{\varepsilon }, \\ 
\left( u^{\varepsilon }\right) ^{+} & = & \left( u^{\varepsilon }\right) ^{-}
& \text{on }\Gamma _{2}, \\ 
\nu \left( \dfrac{\partial u^{\varepsilon }}{\partial n}\right) ^{+} & = & 
\nu \varepsilon \left( \dfrac{\partial u^{\varepsilon }}{\partial n}\right)
^{-} & \text{on }\Gamma _{2}, \\ 
u^{\varepsilon } & = & 0 & \text{on }\partial \Omega _{\varepsilon },%
\end{array}%
\right.  \label{Peps}
\end{equation}%
where the superscript $+$ (resp. $-$) denotes the trace seen from $\Omega $
(resp. from $\Sigma _{\varepsilon }$) on $\Gamma _{2}$. The thin layer $%
\Sigma _{\varepsilon }$ is here considered as an unstable thin boundary
layer whose Reynolds' number $R_{\varepsilon }$ is of order $1/\varepsilon $
(see \cite[pages 239-240]{Lan-Lif}, where Reynolds' number is allowed to
depend on the thinness of the layer). In the problem (\ref{Peps}), we
suppose that the density $f$ of volumic forces belongs to $\mathbf{L}%
^{\infty }\left( \mathbf{R}^{N},\mathbf{R}^{N}\right) $.

Our purpose is to describe the asymptotic behavior of the solution $%
u^{\varepsilon }$ of (\ref{Peps}) when $\varepsilon $ goes to $0$, in order
to derive the Navier wall law. We use $\Gamma $-convergence arguments (see 
\cite{Dal1} for the definition and the properties of the $\Gamma $%
-convergence) in order to characterize the limit problem. Our approach is
based on the tools developed in \cite{Ace-But}, \cite{But-Dal-Mos}, \cite%
{Dal-Mos1}, \cite{Dal-Mos2} and \cite{Dal-Def-Vit}. On $\Gamma _{2}$, we
will get a general Navier law of the kind%
\[
\left\{ 
\begin{array}{rll}
\left( Id-n\otimes n\right) \nu \dfrac{\partial u}{\partial n}+\mu ^{\bullet
}u & = & 0, \\ 
u\cdot n & = & 0,%
\end{array}%
\right.
\]%
where $\mu ^{\bullet }$ is a symmetric matrix $\left( \mu _{ij}\right)
_{i,j=1,\ldots ,N}$ of Borel measures having their support contained in $%
\Gamma _{2}$, which do not charge the polar subsets of $\mathbf{R}^{N}$ and
which satisfy $\mu _{ij}\left( B\right) \zeta _{i}\zeta _{j}\geq 0$,\ $%
\forall \zeta \in \mathbf{R}^{N}$, $\forall B\in \mathcal{B}\left( \mathbf{R}%
^{N}\right) $, where $\mathcal{B}\left( \mathbf{R}^{N}\right) $ denotes the
set of all Borel subsets of $\mathbf{R}^{N}$ and where we have used the
summation convention with respect to repeated indices.

As a first special case, we prove that when $\Omega \subset \left\{
x_{3}>0\right\} $, $\Gamma _{2}=\partial \Omega \cap \left\{ x_{3}=0\right\} 
$ and%
\[
\Sigma _{\varepsilon }=\left\{ x\in \mathbf{R}^{3}\mid x^{\prime }=\left(
x_{1},x_{2}\right) \in \Gamma _{2}\text{, }-\varepsilon h\left( \frac{%
x^{\prime }}{\varepsilon }\right) <x_{3}<0\right\} ,
\]%
where $h$ is a periodic function, we get on $\Gamma _{2}$ the Robin type
boundary conditions%
\[
\left\{ 
\begin{array}{rll}
\dfrac{\partial u_{1}}{\partial x_{3}}\left( x^{\prime },0\right) & = & 
-c_{1}u_{1}\left( x^{\prime },0\right) , \\ 
\dfrac{\partial u_{2}}{\partial x_{3}}\left( x^{\prime },0\right) & = & 
-c_{2}u_{2}\left( x^{\prime },0\right) , \\ 
u_{3}\left( x^{\prime },0\right) & = & 0,%
\end{array}%
\right.
\]%
where $c_{m}$, $m=1,2$, are constants which will be computed in terms of the
solution of appropriate local thin layer problems (\ref{cm}). This situation
can be generalized to the case of a general open and bounded set $\Omega $,
surrounded on a part of its boundary by such a rough thin layer.

As a second example, we will consider the case where%
\[
\Sigma _{\varepsilon }=\left\{ s+tn\left( s\right) \mid s\in \Gamma _{2}%
\text{, }-\varepsilon h\left( s\right) <x_{3}<0\right\} ,
\]%
where $h$ is a Lipschitz continuous and positive function on $\Gamma _{2}$.\
We here prove that Navier's law takes the following expression on $\Gamma
_{2}$%
\[
\left\{ 
\begin{array}{rll}
\left( Id-n\otimes n\right) \dfrac{\partial u}{\partial n}+\dfrac{1}{h}u & =
& 0, \\ 
u\cdot n & = & 0.%
\end{array}%
\right.
\]

In the last part of this work, we consider an optimal control problem.
Choosing $m>0$, we consider the set $\Xi _{m}$ of all the matrices $\mathbf{h%
}=Diag\left( h_{i}\right) _{i=1,..,N}$ of functions $h_{i}:\Gamma
_{2}\rightarrow \left[ 0,+\infty \right] $, which are $d\Gamma _{2}$%
-measurable and satisfy $\int\nolimits_{\Gamma _{2}}h_{i}d\Gamma _{2}=m$, $%
\forall i=1,\ldots ,N$. We suppose that $\Omega $ is smooth enough and
consider the following problem with Navier conditions on $\Gamma _{2}$%
\begin{equation}
\left\{ 
\begin{array}{rrrl}
-\nu \Delta u^{h}+\left( u^{h}\cdot \nabla \right) u^{h}+\nabla p^{h} & = & f
& \text{in }\Omega , \\ 
\limfunc{div}\left( u^{h}\right) & = & 0 & \text{in }\Omega , \\ 
\mathbf{h}\left( Id-n\otimes n\right) \dfrac{\partial u^{h}}{\partial n}%
+u^{h} & = & 0 & \text{on }\Gamma _{2}, \\ 
u^{h}\cdot n & = & 0 & \text{on }\Gamma _{2}.%
\end{array}%
\right.  \label{Ph}
\end{equation}

Let $\left( u^{h},p^{h}\right) $ be the solution of (\ref{Ph}) and define
the functional $\mathbf{F}$ through%
\[
\mathbf{F}\left( \mathbf{h},u\right) =\left\{ 
\begin{array}{ll}
\dfrac{\nu }{2}\dint_{\Omega }\left\vert \nabla u\right\vert ^{2}dx+\dfrac{1%
}{2}\underset{i=1}{\overset{N}{\dsum }}\dint_{\Gamma _{2}}\dfrac{\left(
u_{i}\right) ^{2}}{h_{i}}d\Gamma _{2} &  \\ 
\qquad +\dint_{\Omega }\left( u^{h}\cdot \nabla \right) u^{h}\cdot
vdx-\dint_{\Omega }f\cdot udx & \text{if }u\in \mathbf{V}_{0,\Gamma
_{1}}\left( \Omega \right) , \\ 
+\infty & \text{otherwise,}%
\end{array}%
\right.
\]%
where $\mathbf{V}_{0,\Gamma _{1}}\left( \Omega \right) $ is the functional
space defined in (\ref{spaces}).\ We consider the optimal control problem%
\begin{equation}
\underset{\mathbf{h}\in \Xi _{m}}{\min }\underset{\text{ }u\in \mathbf{V}%
_{0,\Gamma _{1}}\left( \Omega \right) }{\min }\mathbf{F}\left( \mathbf{h}%
,u\right) .  \label{Popm}
\end{equation}

In the last section of this work, we describe the asymptotic behavior of the
solution of (\ref{Popm}), when $m$ goes to $0$, and characterize the zones
where some thin boundary layer appears. A problem of this kind has been
considered in \cite{Esp-Rie}, but for a linear diffusion problem.

\section{Functional framework}

We define the ($H^{1}\left( \mathbf{R}^{N}\right) $) capacity of any compact
subset $K$ of $\mathbf{R}^{N}$ as%
\[
\begin{array}{l}
Cap\left( K\right) \\ 
\qquad =\inf \left\{ \dint\nolimits_{\mathbf{R}^{N}}\left\vert \nabla
\varphi \right\vert ^{2}dx+\dint\nolimits_{\mathbf{R}^{N}}\left\vert \varphi
\right\vert ^{2}dx\mid \varphi \in \mathbf{C}_{c}^{\infty }\left( \mathbf{R}%
^{N}\right) \text{, }\varphi \geq 1\text{ on }K\right\} .%
\end{array}%
\]

If $U$ is an open subset of $\mathbf{R}^{N}$, then we define%
\[
Cap\left( U\right) =\sup \left\{ Cap\left( K\right) \mid K\subset U\text{, }K%
\text{ compact}\right\} .
\]

If $B\subset \mathbf{R}^{N}$ is a Borel subset of $\mathbf{R}^{N}$, then we
define%
\[
Cap\left( B\right) =\inf \left\{ Cap\left( U\right) \mid B\subset U\text{, }U%
\text{ open}\right\} .
\]

\begin{definition}
Let $\mathcal{B}\left( \mathbf{R}^{N}\right) $ be the $\sigma $-algebra of
all Borel subsets of $\mathbf{R}^{N}$.

\begin{enumerate}
\item A property is said to be true quasi-everywhere (q.e.) on $B\in 
\mathcal{B}\left( \mathbf{R}^{N}\right) $ if it is true except on a subset
of $B$ of capacity $Cap$ equal to 0.

\item A function $u:B\rightarrow \overline{\mathbf{R}}$, with $B\in \mathcal{%
B}\left( \mathbf{R}^{N}\right) $, is quasi-continuous on $B$ if, for every $%
\varepsilon >0$, there exists an open subset $U\subset B$ with $Cap\left(
U\right) <\varepsilon $ and such that the restriction of $u$ on $B\setminus
U $ is continuous.

\item Every function $u\in \mathbf{H}^{1}\left( \mathbf{R}^{N}\right) $ has
a quasi-continuous representative $\widetilde{u}$, which is unique for the
equality quasi-everywhere in $\mathbf{R}^{N}$, (see \cite{Zie}, for
example). $\widetilde{u}$\ is given through%
\[
\widetilde{u}\left( x\right) =\underset{r\rightarrow 0^{+}}{\lim }\frac{1}{%
\left\vert B\left( x,r\right) \right\vert }\dint\nolimits_{B\left(
x,r\right) }u\left( y\right) dy,
\]%
for q.e.\ $x\in \mathbf{R}^{N}$, where $\left\vert B\left( x,r\right)
\right\vert $ is the Lebesgue measure of the ball $B\left( x,r\right) $ of $%
\mathbf{R}^{N}$ of radius $r>0$ and centered at $x$.
\end{enumerate}
\end{definition}

We define some notions concerning families of subsets of $\mathbf{R}^{N}$.

\begin{definition}
\begin{enumerate}
\item A subset $\mathcal{D\subset B}\left( \mathbf{R}^{N}\right) $ is a
dense family in $\mathcal{B}\left( \mathbf{R}^{N}\right) $ if, for every $%
A,B\in \mathcal{B}\left( \mathbf{R}^{N}\right) $ with $\overline{A}\subset 
\overset{o}{B}$, there exists $D\in \mathcal{D}$ such that: $\overline{A}%
\subset \overset{o}{D}\subset \overline{D}\subset \overset{o}{B}$, where $%
\overset{o}{A}$ (resp. $\overline{A}$) denotes the interior (resp. the
closure) of $A$.

\item A subset $\mathcal{R\subset B}\left( \mathbf{R}^{N}\right) $ is a rich
family in $\mathcal{B}\left( \mathbf{R}^{N}\right) $ if, for every family $%
\left( A_{t}\right) _{t\in \left] 0,1\right[ }\subset \mathcal{B}\left( 
\mathbf{R}^{N}\right) $ such that $\overline{A}_{s}\subset \overset{o}{A_{t}}
$, for every $s<t$, the set $\left\{ t\in \left] 0,1\right[ \mid A_{t}\notin 
\mathcal{R}\right\} $ is at most countable.
\end{enumerate}
\end{definition}

Let $\mathcal{O}\left( \mathbf{R}^{N}\right) $ be the set of all open
subsets of $\mathbf{R}^{N}$. We consider the class $\mathbb{F}$ of
functionals $F$ from $\mathbf{H}^{1}\left( \mathbf{R}^{N},\mathbf{R}%
^{N}\right) \times \mathcal{O}\left( \mathbf{R}^{N}\right) $ to $\left[
0,+\infty \right] $ satisfying:

\begin{enumerate}
\item[i)] (\textit{Lower semi-continuity}): for every open subset $\omega
\in \mathcal{O}\left( \mathbf{R}^{N}\right) $, the functional $u\mapsto
F\left( u,\omega \right) $ is lower semi-continuous with respect to the
strong topology of $\mathbf{H}^{1}\left( \mathbf{R}^{N},\mathbf{R}%
^{N}\right) $;

\item[ii)] \textit{(Measure property)}: for every $u\in \mathbf{H}^{1}\left( 
\mathbf{R}^{N},\mathbf{R}^{N}\right) $, $\omega \mapsto F\left( u,\omega
\right) $ is the restriction to $\mathcal{O}\left( \mathbf{R}^{N}\right) $
of some Borel measure still denoted $F\left( u,\omega \right) $;

\item[iii)] (\textit{Localization)}: for every $\omega \in \mathcal{O}\left( 
\mathbf{R}^{N}\right) $ and every $u,v\in \mathbf{H}^{1}\left( \mathbf{R}%
^{N},\mathbf{R}^{N}\right) $:%
\[
u_{\mid _{\omega }}=v_{\mid _{\omega }}\Rightarrow F\left( u,\omega \right)
=F\left( v,\omega \right) ;
\]

\item[iv)] ($\mathbf{C}^{1}$\textit{-convexity}): for every $\omega \in 
\mathcal{O}\left( \mathbf{R}^{N}\right) $, the functional $u\mapsto F\left(
u,\omega \right) $ is convex on $\mathbf{H}^{1}\left( \mathbf{R}^{N},\mathbf{%
R}^{N}\right) $ and moreover%
\[
\forall \varphi \in \mathbf{C}^{1}\left( \mathbf{R}^{N}\right) \text{, }%
0\leq \varphi \leq 1:F\left( \varphi u+\left( 1-\varphi \right) v,\omega
\right) \leq F\left( u,\omega \right) +F\left( v,\omega \right) \text{.}
\]
\end{enumerate}

\begin{example}
Let us define $\Gamma _{2,\varepsilon }=\partial \Omega _{\varepsilon }\cap 
\overline{\Sigma _{\varepsilon }}$, for some thin layer $\Sigma
_{\varepsilon }$, as defined above. We consider the functional $%
F^{\varepsilon }$ defined on the space $\mathbf{H}^{1}\left( \mathbf{R}^{N},%
\mathbf{R}^{N}\right) \times \mathcal{O}\left( \mathbf{R}^{N}\right) $
through%
\begin{equation}
F^{\varepsilon }\left( u,\omega \right) =\left\{ 
\begin{array}{ll}
0 & \text{if }\widetilde{u}=0\text{, q.e. on }\Gamma _{2,\varepsilon }\cap
\omega , \\ 
+\infty & \text{otherwise.}%
\end{array}%
\right.  \label{Fepse}
\end{equation}

One can prove that $F^{\varepsilon }$ belongs to $\mathbb{F}$, for every $%
\varepsilon >0$.
\end{example}

Let us set the following definitions.

\begin{definition}
Let $Cap$ be the above-defined capacity.

\begin{enumerate}
\item A\ Borel measure $\lambda $ is absolutely continuous with respect to
the capacity $Cap$ if%
\[
\forall B\in \mathcal{B}\left( \mathbf{R}^{N}\right) :Cap\left( B\right)
=0\Rightarrow \lambda \left( B\right) =0.
\]

\item $\mathcal{M}_{0}$ is the set of nonnegative Borel measures $\mathbf{R}%
^{N}$ which are absolutely continuous with respect to the capacity $Cap$.
\end{enumerate}
\end{definition}

We have the following example of measure in $\mathcal{M}_{0}$.

\begin{example}
For every $E\subset \mathbf{R}^{N}$ such that $Cap\left( E\right) >0$, we
define the measure $\infty _{E}$ through%
\[
\infty _{E}\left( B\right) =\left\{ 
\begin{array}{ll}
0 & \text{if }Cap\left( B\cap E\right) =0, \\ 
+\infty & \text{otherwise.}%
\end{array}%
\right.
\]

Then $\infty _{E}\in \mathcal{M}_{0}$.

Notice that, for every $u\in \mathbf{H}^{1}\left( \mathbf{R}^{N},\mathbf{R}%
^{N}\right) $ and every $\omega \in \mathcal{O}\left( \mathbf{R}^{N}\right) $%
, the functional $F^{\varepsilon }$ defined in (\ref{Fepse}) can be written
as%
\[
F^{\varepsilon }\left( u,\omega \right) =\dint\nolimits_{\omega }\left\vert 
\widetilde{u}\right\vert ^{2}d\infty _{\Gamma _{2,\varepsilon
}}=\dint\nolimits_{\omega }\left\vert u\right\vert ^{2}d\infty _{\Gamma
_{2,\varepsilon }}.
\]
\end{example}

One has the following representation theorem for the functionals of $\mathbb{%
F}$.

\begin{theorem}
\label{theun}(see \cite{Dal-Def-Vit}) For every $F\in \mathbb{F}$, there
exist a finite measure $\lambda \in \mathcal{M}_{0}$, a nonnegative Borel
measure $\nu $ and a Borel function $g:\mathbf{R}^{N}\times \mathbf{R}%
^{N}\rightarrow \left[ 0,+\infty \right] $, with $\zeta \mapsto g\left(
x,\zeta \right) $ convex and lower semi-continuous on $\mathbf{R}^{N}$, such
that%
\[
\forall u\in \mathbf{H}^{1}\left( \mathbf{R}^{N},\mathbf{R}^{N}\right) \text{%
, }\forall \omega \in \mathcal{O}\left( \mathbf{R}^{N}\right) :F\left(
u,\omega \right) =\dint\nolimits_{\omega }g\left( x,\widetilde{u}\left(
x\right) \right) d\lambda +\nu \left( \omega \right) .
\]
\end{theorem}

Throughout the paper, we will need the following Corollary (see \cite[%
Corollary 8.4]{Dal-Def-Vit}).

\begin{corollary}
\label{deux}Let $F\in \mathbb{F}$. If $F\left( .,\omega \right) $ is
quadratic for every $\omega \in \mathcal{O}\left( \mathbf{R}^{N}\right) $,
there exist $\lambda \in \mathcal{M}_{0}$ finite, a symmetric matrix $\left(
a_{ij}\right) _{i,j=1,..,N}$, of Borel functions from $\mathbf{R}^{N}$ to $%
\mathbf{R}$ satisfying $a_{ij}\left( x\right) \zeta _{i}\zeta _{j}\geq 0$, $%
\forall \zeta \in \mathbf{R}^{N}$ and for q.e. $x\in \mathbf{R}^{N}$, for
every $x\in \mathbf{R}^{N}$ a subspace $\mathbf{V}\left( x\right) $ of $%
\mathbf{R}^{N}$, such that, for every $u\in $ $\mathbf{H}^{1}\left( \mathbf{R%
}^{N},\mathbf{R}^{N}\right) $ and every $\omega \in \mathcal{O}\left( 
\mathbf{R}^{N}\right) $:

\begin{enumerate}
\item[a)] if $F\left( u,\omega \right) <+\infty $, then $u\left( x\right)
\in \mathbf{V}\left( x\right) $, for q.e. $x\in \omega $,

\item[b)] if $u\left( x\right) \in \mathbf{V}\left( x\right) $, for q.e. $%
x\in \omega $%
\begin{equation}
F\left( u,\omega \right) =\dint\nolimits_{\omega }a_{ij}u_{i}u_{j}d\lambda .
\label{F1}
\end{equation}
\end{enumerate}
\end{corollary}

\begin{remark}
\label{un}Let $F\in \mathbb{F}$, $\lambda \in \mathcal{M}_{0}$ be the
associated measure and $\Lambda $ be the set defined as $\Lambda =\cup
_{\omega \in A\left( F\right) }\omega $, where%
\[
A\left( F\right) =\left\{ \omega \in \mathcal{O}\left( \mathbf{R}^{N}\right)
\mid F\left( .,\omega \right) <+\infty \text{, for q.e. }x\in \omega
\right\} .
\]

We define the matrix $\mu ^{\bullet }=\left( \mu _{ij}\right) =\left(
a_{ij}\lambda \right) _{i,j=1,..,N}+\infty _{\mathbf{R}^{N}\backslash
\Lambda }Id$ of measures, and, for every $x\in \mathbf{R}^{N}$, the subspace 
$\mathbf{V}\left( x\right) $ through%
\begin{equation}
\mathbf{V}\left( x\right) =\left\{ 
\begin{array}{ll}
\mathbf{R}^{N} & \text{if }x\in \Lambda , \\ 
\left\{ 0\right\} & \text{if }x\in \mathbf{R}^{N}\backslash \Lambda .%
\end{array}%
\right.  \label{Vx}
\end{equation}

For every $u\in \mathbf{H}^{1}\left( \mathbf{R}^{N},\mathbf{R}^{N}\right) $
and every $\omega \in \mathcal{O}\left( \mathbf{R}^{N}\right) $, one has,
using the preceding definition of $\mu ^{\bullet }$%
\[
\dint\nolimits_{\omega }u_{i}u_{j}d\mu _{ij}=\left\{ 
\begin{array}{ll}
\dint_{\omega }a_{ij}u_{i}u_{j}d\lambda & \text{if }\omega \subset \Lambda ,
\\ 
\dint_{\omega \cap \Lambda }a_{ij}u_{i}u_{j}d\lambda & \text{if }\left\{ 
\begin{array}{l}
u\left( x\right) =0\text{, }\forall x\in \omega \cap \mathbf{R}%
^{N}\backslash \Lambda \\ 
\quad \text{and }Cap\left( \omega \cap \mathbf{R}^{N}\backslash \Lambda
\right) >0,%
\end{array}%
\right. \\ 
+\infty & \text{otherwise.}%
\end{array}%
\right.
\]

Thanks to (\ref{Vx}), this expression can be written as%
\[
\dint\nolimits_{\omega }u_{i}u_{j}d\mu _{ij}=\left\{ 
\begin{array}{ll}
\dint_{\omega }a_{ij}u_{i}u_{j}d\lambda & \text{if }u\left( x\right) \in 
\mathbf{V}\left( x\right) \text{, for q.e. }x\in \omega , \\ 
+\infty & \text{otherwise.}%
\end{array}%
\right.
\]

We can thus write the functional $F$ defined in (\ref{F1}) as%
\[
F\left( u,\omega \right) =\dint\nolimits_{\omega }u_{i}u_{j}d\mu _{ij}\text{.%
}
\]
\end{remark}

\section{Study of the problem (\protect\ref{Peps})}

We here suppose that the "outer" boundary $\Gamma _{2,\varepsilon }$ of $%
\Sigma _{\varepsilon }$ can be defined as%
\[
\Gamma _{2,\varepsilon }=\left\{ \left( s,t\right) \mid s\in \Gamma _{2}%
\text{, }t=-\varepsilon h_{\varepsilon }\left( s\right) \right\} ,
\]%
where $h_{\varepsilon }$ is a locally Lipschitz continuous function
satisfying%
\[
\left\Vert h_{\varepsilon }\right\Vert _{\mathbf{L}^{\infty }\left( \Gamma
_{2}\right) }\leq C\text{, }\forall \varepsilon >0,
\]%
for some constant $C$ independent of $\varepsilon $. The Lipschitz
continuity of $h_{\varepsilon }$ ensures the almost everywhere existence of
a unit outer normal vector to $\Gamma _{2,\varepsilon }$, thanks to
Rademacher's Theorem, and ensures the existence of an extension of every
function of $\mathbf{H}^{1}\left( \Omega _{\varepsilon },\mathbf{R}%
^{N}\right) $ in a function of $\mathbf{H}^{1}\left( \mathbf{R}^{N},\mathbf{R%
}^{N}\right) $. Let us define the functional spaces%
\begin{equation}
\begin{array}{rll}
\mathbf{L}^{2}\left( \mathbf{R}^{N},\limfunc{div}\right) & = & \left\{ u\in 
\mathbf{L}^{2}\left( \mathbf{R}^{N},\mathbf{R}^{N}\right) \mid \limfunc{div}%
\left( u\right) =0\text{ in }\mathbf{R}^{N}\right\} , \\ 
\mathbf{H}_{\Gamma _{1}}^{1}\left( \mathbf{R}^{N},\limfunc{div}\right) & = & 
\left\{ 
\begin{array}{r}
u\in \mathbf{H}^{1}\left( \mathbf{R}^{N},\mathbf{R}^{N}\right) \mid \limfunc{%
div}\left( u\right) =0\text{ in }\mathbf{R}^{N}\text{,\quad } \\ 
u=0\text{ on }\Gamma _{1}%
\end{array}%
\right\} , \\ 
\mathbf{H}_{\Gamma _{1}}^{1}\left( \Omega ,\limfunc{div}\right) & = & 
\left\{ u\in \mathbf{H}^{1}\left( \Omega ,\mathbf{R}^{N}\right) \mid 
\limfunc{div}\left( u\right) =0\text{ in }\Omega \text{, }u=0\text{ on }%
\Gamma _{1}\right\} , \\ 
\mathbf{V}_{\Gamma _{1}}\left( \Omega \right) & = & \mathbf{L}^{2}\left( 
\mathbf{R}^{N},\limfunc{div}\right) \cap \mathbf{H}_{\Gamma _{1}}^{1}\left(
\Omega ,\limfunc{div}\right) , \\ 
\mathbf{V}_{0,\Gamma _{1}}\left( \Omega \right) & = & \mathbf{H}_{\Gamma
_{1}}^{1}\left( \Omega ,\limfunc{div}\right) \cap \left\{ u\in \mathbf{H}%
^{1}\left( \Omega ,\mathbf{R}^{N}\right) \mid u\cdot n=0\text{ on }\Gamma
_{2}\right\} .%
\end{array}
\label{spaces}
\end{equation}

In (\ref{Peps}), let us replace throughout this section the homogeneous
Dirichlet boundary condition $u^{\varepsilon }=0$, on $\partial \Omega
_{\varepsilon }$ by a combination between the homogeneous Dirichlet boundary
condition $u^{\varepsilon }=0$, on $\Gamma _{2,\varepsilon }\cap \omega $,
for a given $\omega \in \mathcal{O}\left( \mathbf{R}^{N}\right) $, and
homogeneous Neumann boundary conditions on $\Gamma _{2,\varepsilon
}\setminus \left( \Gamma _{2,\varepsilon }\cap \omega \right) $. We
introduce the functional space adpated to (\ref{Peps}), with these modified
boundary conditions%
\[
\mathbf{V}_{0,\omega }\left( \Omega _{\varepsilon }\right) =\left\{ 
\begin{array}{r}
v\in \mathbf{H}^{1}\left( \Omega _{\varepsilon },\mathbf{R}^{N}\right) \mid 
\limfunc{div}\left( v\right) =0\text{ in }\Omega _{\varepsilon }\text{%
,\qquad } \\ 
v=0\text{ on }\Gamma _{1}\cup \left( \Gamma _{2,\varepsilon }\cap \omega
\right)%
\end{array}%
\right\} .
\]

The variational formulation of (\ref{Peps}) can be written as%
\begin{equation}
\begin{array}{l}
\forall \varphi \in \mathbf{V}_{0,\omega }\left( \Omega _{\varepsilon
}\right) :\nu \dint\nolimits_{\Omega }\nabla u^{\varepsilon }\cdot \nabla
\varphi dx+\nu \varepsilon \dint\nolimits_{\Sigma _{\varepsilon }}\nabla
u^{\varepsilon }\cdot \nabla \varphi dx \\ 
\qquad +\dint\nolimits_{\Omega _{\varepsilon }}\left( u^{\varepsilon }\cdot
\nabla \right) u^{\varepsilon }\cdot \varphi dx=\dint\nolimits_{\Omega
_{\varepsilon }}f\cdot \varphi dx.%
\end{array}
\label{varform}
\end{equation}

Thanks to \cite{Tem}, for example, we deduce that (\ref{Peps}) has a unique
solution $\left( u^{\varepsilon },p^{\varepsilon }\right) $ belonging to the
space $\mathbf{V}_{0,\omega }\left( \Omega _{\varepsilon }\right) \times 
\mathbf{L}^{2}\left( \Omega _{\varepsilon }\right) /\mathbf{R}$.

\begin{proposition}
\label{estim}The solution $\left( u^{\varepsilon },p^{\varepsilon }\right) $
of (\ref{Peps}) satisfies the following estimates%
\[
\begin{array}{rll}
\underset{\varepsilon }{\sup }\left( \dint_{\Omega }\left\vert \nabla
u^{\varepsilon }\right\vert ^{2}dx+\varepsilon \dint_{\Sigma _{\varepsilon
}}\left\vert \nabla u^{\varepsilon }\right\vert ^{2}dx\right) & < & +\infty ,
\\ 
\underset{\varepsilon }{\sup }\dint_{\mathbf{R}^{N}}\left\vert
u^{\varepsilon }\right\vert ^{2}dx & < & +\infty , \\ 
\underset{\varepsilon }{\sup }\left\Vert p^{\varepsilon }\right\Vert _{%
\mathbf{L}^{2}\left( \Omega _{\varepsilon }\right) /\mathbf{R}} & < & 
+\infty .%
\end{array}%
\]
\end{proposition}

\begin{proof}
1. Taking $u^{\varepsilon }$ as test-function in (\ref{varform}), we obtain%
\[
\begin{array}{l}
\nu \dint_{\Omega }\left\vert \nabla u^{\varepsilon }\right\vert ^{2}dx+\nu
\varepsilon \dint_{\Sigma _{\varepsilon }}\left\vert \nabla u^{\varepsilon
}\right\vert ^{2}dx \\ 
\qquad 
\begin{array}{ll}
= & \dint_{\Omega }f\cdot u^{\varepsilon }dx+\dint_{\Sigma _{\varepsilon
}}f\cdot u^{\varepsilon }dx \\ 
\leq & \left\Vert f\right\Vert _{\mathbf{L}^{\infty }\left( \mathbf{R}^{N},%
\mathbf{R}^{N}\right) }\left\Vert u^{\varepsilon }\right\Vert _{\mathbf{L}%
^{1}\left( \Omega _{\varepsilon },\mathbf{R}^{N}\right) } \\ 
\leq & \left\Vert f\right\Vert _{\mathbf{L}^{\infty }\left( \mathbf{R}^{N},%
\mathbf{R}^{N}\right) }C\left( \Omega \right) \left\Vert \nabla
u^{\varepsilon }\right\Vert _{\mathbf{L}^{1}\left( \Omega ,\mathbf{R}%
^{N}\right) },%
\end{array}%
\end{array}%
\]%
using Poincar\'{e}'s inequality. Cauchy-Schwarz' inequality implies%
\[
\begin{array}{l}
\dint\nolimits_{\Omega }\left\vert \nabla u^{\varepsilon }\right\vert
^{2}dx+\varepsilon \dint\nolimits_{\Sigma _{\varepsilon }}\left\vert \nabla
u^{\varepsilon }\right\vert ^{2}dx \\ 
\qquad \leq C\left( f,\Omega \right) \left( \left( \dint\nolimits_{\Omega
}\left\vert \nabla u^{\varepsilon }\right\vert ^{2}dx\right) ^{1/2}+\left(
\varepsilon \dint\nolimits_{\Sigma _{\varepsilon }}\left\vert \nabla
u^{\varepsilon }\right\vert ^{2}dx\right) ^{1/2}\right) ,%
\end{array}%
\]%
whence, using the trivial inequality $\left( a+b\right) ^{2}\leq 2\left(
a^{2}+b^{2}\right) $%
\[
\dint\nolimits_{\Omega }\left\vert \nabla u^{\varepsilon }\right\vert
^{2}dx+\varepsilon \dint\nolimits_{\Sigma _{\varepsilon }}\left\vert \nabla
u^{\varepsilon }\right\vert ^{2}dx\leq C\Rightarrow \left\Vert \nabla
u^{\varepsilon }\right\Vert _{\mathbf{L}^{1}\left( \Omega _{\varepsilon },%
\mathbf{R}^{N}\right) }\leq C.
\]

The continuous embedding from $\mathbf{W}_{\Gamma _{1}}^{1,1}\left( \Omega
_{\varepsilon },\mathbf{R}^{N}\right) $ to $\mathbf{L}^{2}\left( \Omega
_{\varepsilon },\mathbf{R}^{N}\right) $ implies the existence of a constant $%
C$ independent of $\varepsilon $ such that%
\[
\dint\nolimits_{\Omega _{\varepsilon }}\left\vert u^{\varepsilon
}\right\vert ^{2}dx\leq C.
\]

\noindent 2. Let us define the zero mean value pressure $\overline{%
p^{\varepsilon }}=p^{\varepsilon }-\frac{1}{\left\vert \Omega _{\varepsilon
}\right\vert }\int_{\Omega _{\varepsilon }}p^{\varepsilon }dx$, and let $%
\psi _{\varepsilon }$ be the solution of the following problem (see \cite%
{Tem})%
\begin{equation}
\left\{ 
\begin{array}{rlll}
\limfunc{div}\left( \psi _{\varepsilon }\right) & = & \overline{%
p^{\varepsilon }} & \text{in }\Omega _{\varepsilon }, \\ 
\psi _{\varepsilon } & = & 0 & \text{on }\Gamma _{1}\cup \left( \Gamma
_{2}\cap \omega \right) , \\ 
\left\Vert \nabla \psi _{\varepsilon }\right\Vert _{\mathbf{L}^{2}\left(
\Omega _{\varepsilon },\mathbf{R}^{N^{2}}\right) } & \leq & C\left( \Omega
\right) \left\Vert \overline{p^{\varepsilon }}\right\Vert _{\mathbf{L}%
^{2}\left( \Omega _{\varepsilon }\right) }, & 
\end{array}%
\right.  \label{psieps}
\end{equation}%
for some constant $C\left( \Omega \right) $ independent of $\varepsilon $.
Multiplying (\ref{Peps})$_{1,2}$ by $\psi _{\varepsilon }$ and using Green's
formula, one obtains%
\[
\begin{array}{l}
\nu \dint\nolimits_{\Omega }\nabla u^{\varepsilon }\cdot \nabla \psi
_{\varepsilon }dx+\nu \varepsilon \dint\nolimits_{\Sigma _{\varepsilon
}}\nabla u^{\varepsilon }\cdot \nabla \psi _{\varepsilon
}dx+\dint\nolimits_{\Omega _{\varepsilon }}\left( u^{\varepsilon }\cdot
\nabla \right) u^{\varepsilon }\cdot \psi _{\varepsilon }dx \\ 
\qquad =\dint\nolimits_{\Omega _{\varepsilon }}f\cdot \psi _{\varepsilon
}dx+\dint\nolimits_{\Omega _{\varepsilon }}\left( \overline{p^{\varepsilon }}%
\right) ^{2}dx.%
\end{array}%
\]

Because%
\[
\begin{array}{rll}
\left\vert \dint_{\Omega _{\varepsilon }}f\cdot \psi _{\varepsilon
}dx\right\vert & \leq & \left\Vert f\right\Vert _{\mathbf{L}^{2}\left( 
\mathbf{R}^{N},\mathbf{R}^{N}\right) }\left\Vert \psi _{\varepsilon
}\right\Vert _{\mathbf{L}^{2}\left( \Omega _{\varepsilon },\mathbf{R}%
^{N}\right) } \\ 
& \leq & C\left\Vert \overline{p^{\varepsilon }}\right\Vert _{\mathbf{L}%
^{2}\left( \Omega _{\varepsilon }\right) } \\ 
\left\vert \dint_{\Omega _{\varepsilon }}\left( u^{\varepsilon }\cdot \nabla
\right) u^{\varepsilon }\cdot \psi _{\varepsilon }dx\right\vert & \leq & 
C\left\Vert \psi _{\varepsilon }\right\Vert _{\mathbf{L}^{2}\left( \Omega
_{\varepsilon },\mathbf{R}^{N}\right) }\left\Vert \nabla u^{\varepsilon
}\right\Vert _{\mathbf{L}^{2}\left( \Omega _{\varepsilon },\mathbf{R}%
^{N}\right) }^{2} \\ 
& \leq & C\left\Vert \overline{p^{\varepsilon }}\right\Vert _{\mathbf{L}%
^{2}\left( \Omega _{\varepsilon }\right) }\left\Vert \nabla u^{\varepsilon
}\right\Vert _{\mathbf{L}^{2}\left( \Omega _{\varepsilon },\mathbf{R}%
^{N}\right) }^{2} \\ 
\left\vert \dint_{\Omega }\nabla u^{\varepsilon }\cdot \nabla \psi
_{\varepsilon }dx\right\vert & \leq & C\left\Vert \overline{p^{\varepsilon }}%
\right\Vert _{\mathbf{L}^{2}\left( \Omega _{\varepsilon }\right) }\left\Vert
\nabla u^{\varepsilon }\right\Vert _{\mathbf{L}^{2}\left( \Omega
_{\varepsilon },\mathbf{R}^{N}\right) }, \\ 
\left\vert \dint_{\Sigma _{\varepsilon }}\nabla u^{\varepsilon }\cdot \nabla
\psi _{\varepsilon }dx\right\vert & \leq & C\left\Vert \overline{%
p^{\varepsilon }}\right\Vert _{\mathbf{L}^{2}\left( \Omega _{\varepsilon
}\right) }\left\Vert \nabla u^{\varepsilon }\right\Vert _{\mathbf{L}%
^{2}\left( \Omega _{\varepsilon },\mathbf{R}^{N}\right) },%
\end{array}%
\]%
thanks to (\ref{psieps})$_{3}$ and using Poincar\'{e}'s inequality, we obtain%
\[
\left\Vert \overline{p^{\varepsilon }}\right\Vert _{\mathbf{L}^{2}\left(
\Omega _{\varepsilon }\right) }^{2}\leq C\left( \left\Vert \nabla
u^{\varepsilon }\right\Vert _{\mathbf{L}^{2}\left( \Omega _{\varepsilon },%
\mathbf{R}^{N}\right) }^{2}+1\right) \left\Vert \overline{p^{\varepsilon }}%
\right\Vert _{\mathbf{L}^{2}\left( \Omega _{\varepsilon }\right) },
\]%
which proves the third estimate.
\end{proof}

\begin{remark}
We can observe that, when we impose an homogeneous Dirichlet boundary
condition on the whole $\Gamma _{2,\varepsilon }$, for example when $\omega =%
\mathbf{R}^{N}$, the above estimates can be obtained in a simpler way,
assuming only that $f\in \mathbf{L}^{2}\left( \mathbf{R}^{N},\mathbf{R}%
^{N}\right) $.
\end{remark}

\section{Convergence}

Every function $u\in \mathbf{H}_{\Gamma _{1}}^{1}\left( \Omega _{\varepsilon
},\limfunc{div}\right) $ can be extended in a function of the space $\mathbf{%
H}_{\Gamma _{1}}^{1}\left( \mathbf{R}^{N},\limfunc{div}\right) $, still
denoted $u$ (see \cite[Theorem 4.3.3]{Tri}, for example). We define the
functional $\Phi ^{\varepsilon }$ on $\mathbf{L}^{2}\left( \mathbf{R}^{N},%
\mathbf{R}^{N}\right) $ associated to (\ref{Peps}), with the above-described
modified boundary conditions on $\Gamma _{2,\varepsilon }$ through%
\begin{equation}
\Phi ^{\varepsilon }\left( u\right) =\left\{ 
\begin{array}{ll}
\nu \dint_{\Omega }\left\vert \nabla u\right\vert ^{2}dx+\nu \varepsilon
\dint_{\mathbf{R}^{N}\setminus \mathbf{\Omega }}\left\vert \nabla
u\right\vert ^{2}dx & \text{if }u\in \mathbf{H}_{\Gamma _{1}}^{1}\left( 
\mathbf{R}^{N},\limfunc{div}\right) , \\ 
+\infty & \text{otherwise}%
\end{array}%
\right.  \label{Phieps}
\end{equation}%
and the functional $\Phi ^{0}$ defined on $\mathbf{L}^{2}\left( \mathbf{R}%
^{N},\mathbf{R}^{N}\right) $ through%
\[
\Phi ^{0}\left( u\right) =\left\{ 
\begin{array}{ll}
\nu \dint_{\Omega }\left\vert \nabla u\right\vert ^{2}dx & \text{if }u\in 
\mathbf{V}_{\Gamma _{1}}\left( \Omega \right) , \\ 
+\infty & \text{otherwise.}%
\end{array}%
\right.
\]

From the estimates given in Proposition \ref{estim}, we can deduce that the
asymptotic behaviour of the problem (\ref{Peps}) is obtained when studying
the $\Gamma $-limit of the associated energy functional for the following
topology.

\begin{definition}
A sequence $\left( u_{\varepsilon }\right) _{\varepsilon }$.$\tau $%
-converges to $u$, if it converges to $u$ in the strong topology of $\mathbf{%
L}^{2}\left( \mathbf{R}^{N},\mathbf{R}^{N}\right) $ and if $%
\sup_{\varepsilon }\Phi ^{\varepsilon }\left( u_{\varepsilon }\right)
<+\infty $.
\end{definition}

We first present the $\Gamma $-convergence result for $\left( \Phi
^{\varepsilon }\right) _{\varepsilon }$.

\begin{proposition}
\label{convPhi}When $\varepsilon $ goes to $0$, the sequence $\left( \Phi
^{\varepsilon }\right) _{\varepsilon }$ $\Gamma $-converges to $\Phi ^{0}$,
in the topology $\tau $.
\end{proposition}

\begin{proof}
Step 1: verification of the $\Gamma $-$\lim \sup $. Take any $u\in \mathbf{V}%
_{\Gamma _{1}}\left( \Omega \right) $ and consider the set $\Omega
^{0,\varepsilon }=\Omega \cup \partial \Omega \cup \Sigma ^{0,\varepsilon }$%
, with%
\[
\Sigma ^{0,\varepsilon }=\left\{ x\in \mathbf{R}^{N}\mid 0<d\left(
x,\partial \Omega \right) <\sqrt{\varepsilon }\right\} ,
\]%
where $d\left( x,\partial \Omega \right) $ denotes the euclidean distance
between $x$ and the boundary $\partial \Omega $. Let $u^{1,\varepsilon }$ be
such that $\limfunc{div}\left( u^{1,\varepsilon }\right) =0$ in $\mathbf{R}%
^{N}$ and%
\[
\left\Vert u-u^{1,\varepsilon }\right\Vert _{\mathbf{L}^{2}\left( \mathbf{R}%
^{N}\setminus \Omega ^{0,\varepsilon },\mathbf{R}^{N}\right) }<\varepsilon .
\]

We define $\overline{u}^{1,\varepsilon }$ through%
\[
\overline{u}^{1,\varepsilon }=\left\{ 
\begin{array}{ll}
u^{1,\varepsilon } & \text{in }\mathbf{R}^{N}\setminus \Omega
^{0,\varepsilon }, \\ 
0 & \text{on }\partial \Omega ^{0,\varepsilon }.%
\end{array}%
\right.
\]

We then take a nonnegative and smooth function $\rho _{\varepsilon }\in 
\mathbf{C}_{c}^{\infty }\left( \mathbf{R}^{N}\right) $ with support in $%
B\left( 0,\varepsilon \right) $ and satisfying $\int\nolimits_{\mathbf{R}%
^{N}}\rho _{\varepsilon }\left( x\right) dx=1$. We define the function $%
\overline{u}^{0,\varepsilon }$ through $\overline{u}^{0,\varepsilon }=\left(
\rho _{\varepsilon }\ast \overline{u}^{1,\varepsilon }\right) _{\mid \mathbf{%
R}^{N}\setminus \overline{\Omega ^{0,\varepsilon }}}$. There exists $\overset%
{\smallfrown }{u}\in \mathbf{L}^{2}\left( \mathbf{R}^{N},\mathbf{R}%
^{N}\right) $ such that $\limfunc{curl}(\overset{\smallfrown }{u})=u$ in $%
\mathbf{R}^{N}$ (see \cite{Tem}, for example). We finally define the
function $u^{0,\varepsilon }$ through%
\[
u^{0,\varepsilon }=\left\{ 
\begin{array}{ll}
\overline{u}^{0,\varepsilon } & \text{in }\mathbf{R}^{N}\setminus \overline{%
\Omega ^{0,\varepsilon }}, \\ 
\limfunc{curl}\left( \overset{\smallfrown }{u}\dfrac{\sqrt{\varepsilon }%
-d\left( x,\partial \Omega \right) }{\sqrt{\varepsilon }}\right) & \text{in }%
\Sigma ^{0,\varepsilon }, \\ 
u & \text{in }\Omega .%
\end{array}%
\right.
\]

We immediately satisfy that $u^{0,\varepsilon }\in \mathbf{H}_{\Gamma
_{1}}^{1}\left( \mathbf{R}^{N},\limfunc{div}\right) $, that the sequence $%
\left( u^{0,\varepsilon }\right) _{\varepsilon }$ converges to $u$ in the
topology $\tau $ and that%
\[
\underset{\varepsilon \rightarrow 0}{\lim \sup }\Phi ^{\varepsilon }\left(
u^{0,\varepsilon }\right) \leq \nu \dint\nolimits_{\Omega }\left\vert \nabla
u\right\vert ^{2}dx=\Phi ^{0}\left( u\right) .
\]

\noindent Step 2: verification of the $\Gamma $-$\lim \inf $. We take any
sequence $\left( u_{\varepsilon }\right) _{\varepsilon }$ contained in $%
\mathbf{H}_{\Gamma _{1}}^{1}\left( \mathbf{R}^{N},\limfunc{div}\right) $
which converges to $u$ in the topology $\tau $.\ We trivially have%
\[
\Phi ^{0}\left( u\right) \leq \underset{\varepsilon \rightarrow 0}{\lim \inf 
}\Phi ^{0}\left( u_{\varepsilon }\right) \leq \underset{\varepsilon
\rightarrow 0}{\lim \inf }\Phi ^{\varepsilon }\left( u_{\varepsilon }\right)
,
\]%
thanks to the lower semi-continuity property of $\Phi ^{0}$ for the weak
topology of $\mathbf{H}^{1}\left( \mathbf{R}^{N},\mathbf{R}^{N}\right) $.
\end{proof}

We define the functional $G^{\varepsilon }$ on $\mathbf{L}^{2}\left( \mathbf{%
R}^{N},\mathbf{R}^{N}\right) \times \mathcal{O}\left( \mathbf{R}^{N}\right) $
through%
\[
G^{\varepsilon }\left( u,\omega \right) =\left\{ 
\begin{array}{ll}
\Phi ^{\varepsilon }\left( u\right) +F^{\varepsilon }\left( u,\omega \right)
& \text{if }u\in \mathbf{H}_{\Gamma _{1}}^{1}\left( \mathbf{R}^{N},\limfunc{%
div}\right) , \\ 
+\infty & \text{otherwise,}%
\end{array}%
\right.
\]%
where $F^{\varepsilon }$ is defined in (\ref{Fepse}). Our main result is the
following.

\begin{theorem}
\label{quatre}There exist a rich family $\mathcal{R}\subset \mathcal{B}%
\left( \mathbf{R}^{N}\right) $ and a symmetric matrix $\mu ^{\bullet
}=\left( \mu _{ij}\right) _{i,j=1,\ldots ,N}$ of Borel measures having their
support contained in $\Gamma _{2}$, which are absolutely continuous with
respect to the above-defined capacity $Cap$, and satisfying $\mu _{ij}\left(
B\right) \zeta _{i}\zeta _{j}\geq 0$, $\forall \zeta \in \mathbf{R}^{N}$, $%
\forall B\in \mathcal{B}\left( \mathbf{R}^{N}\right) $, such that, for every 
$u\in \mathbf{V}_{\Gamma _{1}}\left( \Omega \right) $ and every $\omega \in 
\mathcal{R}\cap \mathcal{O}\left( \mathbf{R}^{N}\right) $%
\[
\left( \underset{\varepsilon \rightarrow 0}{\Gamma \text{-}\lim }%
G^{\varepsilon }\right) \left( u,\omega \right) =\nu \dint\nolimits_{\Omega
}\left\vert \nabla u\right\vert ^{2}+\dint\nolimits_{\Gamma _{2}\cap \omega
}u_{i}u_{j}d\mu _{ij}=:G^{0}\left( u,\omega \right) ,
\]%
where the $\Gamma $-limit is taken with respect to the topology $\tau $.
\end{theorem}

\begin{proof}
The upper and lower $\Gamma $-limits of the sequence $\left( G^{\varepsilon
}\right) _{\varepsilon }$, with respect to the topology $\tau $, exist,
which are respectively defined through%
\begin{equation}
\forall u\in \mathbf{V}_{\Gamma _{1}}\left( \Omega \right) \text{, }\forall
B\in \mathcal{B}\left( \mathbf{R}^{N}\right) :\left\{ 
\begin{array}{ccl}
G^{s}\left( u,B\right) & = & \underset{u_{\varepsilon }\overset{\tau }{%
\rightharpoonup }u}{\inf }\underset{\varepsilon \rightarrow 0}{\lim \sup }%
G^{\varepsilon }\left( u_{\varepsilon },B\right) , \\ 
G^{i}\left( u,B\right) & = & \underset{u_{\varepsilon }\overset{\tau }{%
\rightharpoonup }u}{\inf }\underset{\varepsilon \rightarrow 0}{\lim \inf }%
G^{\varepsilon }\left( u_{\varepsilon },B\right) .%
\end{array}%
\right.  \label{27}
\end{equation}

Because $F^{\varepsilon }$ takes nonnegative values and thanks to
Proposition \ref{convPhi}, we observe that, for every $B\in \mathcal{B}%
\left( \mathbf{R}^{N}\right) $, one has%
\[
G^{s}\left( .,B\right) \geq \Phi ^{0}\left( .\right) \text{ ; }G^{i}\left(
.,B\right) \geq \Phi ^{0}\left( .\right) .
\]

Let us define the functionals $F^{s}$ and $F^{i}$ on $\mathbf{L}^{2}\left( 
\mathbf{R}^{N},\mathbf{R}^{N}\right) \times \mathcal{B}\left( \mathbf{R}%
^{N}\right) $ through%
\[
\left( F^{0}\right) ^{\alpha }\left( u,B\right) =\left\{ 
\begin{array}{ll}
G^{\alpha }\left( u,B\right) -\Phi ^{0}\left( u\right) & \text{if }u\in 
\mathbf{V}_{\Gamma _{1}}\left( \Omega \right) , \\ 
+\infty & \text{otherwise,}%
\end{array}%
\right.
\]%
with $\alpha =s,i$. Let $u\in \mathbf{V}_{\Gamma _{1}}\left( \Omega \right) $
and $\left( u_{\varepsilon }\right) _{\varepsilon }\subset \mathbf{H}%
_{\Gamma _{1}}^{1}\left( \mathbf{R}^{N},\limfunc{div}\right) $ be such that $%
\left( u_{\varepsilon }\right) _{\varepsilon }$ converges to $u$ in the
topology $\tau $. We define $z_{\varepsilon }=u_{\varepsilon }-u$.\ Thus $%
\left( z_{\varepsilon }\right) _{\varepsilon }\subset \mathbf{H}^{1}\left( 
\mathbf{R}^{N},\mathbf{R}^{N}\right) $ and $\left( z_{\varepsilon }\right)
_{\varepsilon }$ converges to $0$ in the topology $\tau $. Replacing $%
u_{\varepsilon }$ by $z_{\varepsilon }+u$ in (\ref{27}), one obtains, using
the quadratic property of $\Phi ^{\varepsilon }$%
\[
\begin{array}{lll}
\left( F^{0}\right) ^{s}\left( u,B\right) & = & \underset{z_{\varepsilon }%
\overset{\tau }{\rightarrow }0}{\inf }\underset{\varepsilon \rightarrow 0}{%
\lim \sup }\left( \Phi ^{\varepsilon }\left( z_{\varepsilon }\right)
+F^{\varepsilon }\left( u+z_{\varepsilon },B\right) \right) , \\ 
\left( F^{0}\right) ^{i}\left( u,B\right) & = & \underset{z_{\varepsilon }%
\overset{\tau }{\rightarrow }0}{\inf }\underset{\varepsilon \rightarrow 0}{%
\lim \inf }\left( \Phi ^{\varepsilon }\left( z_{\varepsilon }\right)
+F^{\varepsilon }\left( u+z_{\varepsilon },B\right) \right) .%
\end{array}%
\]

The functionals $\left( F^{0}\right) ^{s}$ and $\left( F^{0}\right) ^{i}$
satisfy the following properties.

\begin{enumerate}
\item For every $u\in \mathbf{V}_{\Gamma _{1}}\left( \Omega \right) $, $%
\left( F^{0}\right) ^{s}\left( u,.\right) $ and $\left( F^{0}\right)
^{i}\left( u,.\right) $ are nonnegative measures, because $F^{\varepsilon
}\left( u+z_{\varepsilon },.\right) $ is a measure for every $\varepsilon >0$
and for every sequence $\left( z_{\varepsilon }\right) _{\varepsilon
}\subset \mathbf{V}_{\Gamma _{1}}\left( \Omega \right) $ which converges to $%
0$ in the topology $\tau $.

\item $\left( F^{0}\right) ^{s}\left( .,B\right) $ and $\left( F^{0}\right)
^{i}\left( .,B\right) $ are lower semi-continuous on $\mathbf{H}^{1}\left( 
\mathbf{R}^{N},\mathbf{R}^{N}\right) $, when equipped with its strong
topology, because $G^{s}\left( .,B\right) $, $G^{i}\left( .,B\right) $ and $%
\Phi ^{0}$ are lower semi-continuous as upper, lower, or $\Gamma $-limits of
functionals which are lower semi-continuous for this strong topology.

\item Let $\omega \in \mathcal{O}\left( \mathbf{R}^{N}\right) $ and $u,v\in 
\mathbf{V}_{\Gamma _{1}}\left( \Omega \right) $ be such that $u_{\mid \omega
}=v_{\mid \omega }$.\ Then $\left( F^{0}\right) ^{s}\left( u,\omega \right)
=\left( F^{0}\right) ^{s}\left( v,\omega \right) $ and $\left( F^{0}\right)
^{i}\left( u,\omega \right) =\left( F^{0}\right) ^{i}\left( v,\omega \right) 
$, because $F^{\varepsilon }\left( u+z_{\varepsilon },\omega \right)
=F^{\varepsilon }\left( v+z_{\varepsilon },\omega \right) $, for every
sequence $\left( z_{\varepsilon }\right) _{\varepsilon }$ such that $%
u+z_{\varepsilon }\in \mathbf{H}_{\Gamma _{1}}^{1}\left( \mathbf{R}^{N},%
\limfunc{div}\right) $, for every $\varepsilon >0$.

\item Take any $\varphi \in \mathbf{C}^{1}\left( \mathbf{R}^{N}\right) $
such that $0\leq \varphi \leq 1$, $u,v\in \mathbf{V}_{\Gamma _{1}}\left(
\Omega \right) $ and $B\in \mathcal{B}\left( \mathbf{R}^{N}\right) $.\ One
has, for every sequence $\left( z_{\varepsilon }\right) _{\varepsilon
}\subset \mathbf{V}_{\Gamma _{1}}\left( \Omega \right) $ converging to $0$
in the topology $\tau $%
\[
\begin{array}{ccl}
F^{\varepsilon }\left( z_{\varepsilon }+\varphi u+\left( 1-\varphi \right)
v,B\right) & = & F^{\varepsilon }\left( \left( z_{\varepsilon }+u\right)
\varphi +\left( 1-\varphi \right) \left( z_{\varepsilon }+v\right) ,B\right)
\\ 
& \leq & F^{\varepsilon }\left( z_{\varepsilon }+u,B\right) +F_{\varepsilon
}\left( z_{\varepsilon }+v,B\right) ,%
\end{array}%
\]%
because $F^{\varepsilon }$ is $\mathbf{C}^{1}$-convex. Because $\Phi
^{\varepsilon }$ takes nonnegative values, for every $\varepsilon >0$, one
has%
\[
\begin{array}{l}
\underset{\varepsilon \rightarrow 0}{\lim \sup }\left( \Phi ^{\varepsilon
}\left( z_{\varepsilon }\right) +F^{\varepsilon }\left( z_{\varepsilon
}+\varphi u+\left( 1-\varphi \right) v,B\right) \right) \\ 
\qquad 
\begin{array}{ll}
\leq & \underset{\varepsilon \rightarrow 0}{\lim \sup }\left( \Phi
^{\varepsilon }\left( z_{\varepsilon }\right) +F^{\varepsilon }\left(
z_{\varepsilon }+u,B\right) +\Phi ^{\varepsilon }\left( z_{\varepsilon
}\right) +F^{\varepsilon }\left( z_{\varepsilon }+v,B\right) \right) \\ 
\leq & \underset{\varepsilon \rightarrow 0}{\lim \sup }\left( \Phi
^{\varepsilon }\left( z_{\varepsilon }\right) +F^{\varepsilon }\left(
z_{\varepsilon }+u,B\right) \right) \\ 
& \qquad +\underset{\varepsilon \rightarrow 0}{\lim \sup }\left( \Phi
^{\varepsilon }\left( z_{\varepsilon }\right) +F^{\varepsilon }\left(
z_{\varepsilon }+v,B\right) \right) .%
\end{array}%
\end{array}%
\]

Taking the infimum over all sequences $\left( z_{\varepsilon }\right)
_{\varepsilon }\subset \mathbf{H}^{1}\left( \mathbf{R}^{N},\mathbf{R}%
^{N}\right) $ which converge to $0$ in the topology $\tau $, one obtains%
\[
\left( F^{0}\right) ^{s}\left( \varphi u+\left( 1-\varphi \right) v,B\right)
\leq \left( F^{0}\right) ^{s}\left( u,B\right) +\left( F^{0}\right)
^{s}\left( v,B\right) .\ 
\]

We prove in a similar way that $\left( F^{0}\right) ^{s}$ is convex. Thus $%
\left( F^{0}\right) ^{s}$ is $\mathbf{C}^{1}$-convex.
\end{enumerate}

Thanks to the compacity theorem of \cite{DeG-Fra}, there exist a subsequence 
$\left( \varepsilon _{k}\right) _{k}$ and a dense and countable family $%
\mathcal{D}\subset \mathcal{B}\left( \mathbf{R}^{N}\right) $ such that, for
every $u\in \mathbf{V}_{\Gamma _{1}}\left( \Omega \right) $ and every $B\in 
\mathcal{D}$%
\[
\left( \underset{k\rightarrow +\infty }{\Gamma \text{-}\lim }G^{\varepsilon
_{k}}\right) \left( u,B\right) =G^{0}\left( u,B\right) ,
\]%
where the $\Gamma $-limit is taken with respect to the topology $\tau $. We
then define the functional $F^{0}$ on $\mathbf{L}^{2}\left( \mathbf{R}^{N},%
\mathbf{R}^{N}\right) \times \mathcal{D}$ as%
\begin{equation}
F^{0}\left( u,B\right) =\left\{ 
\begin{array}{ll}
G^{0}\left( u,B\right) -\Phi ^{0}\left( u\right) & \text{if }u\in \mathbf{V}%
_{\Gamma _{1}}\left( \Omega \right) , \\ 
+\infty & \text{otherwise.}%
\end{array}%
\right.  \label{30}
\end{equation}

We have $F^{0}=\left( F^{0}\right) ^{s}=\left( F^{0}\right) ^{i}$ on $%
\mathbf{L}^{2}\left( \mathbf{R}^{N},\mathbf{R}^{N}\right) \times \mathcal{D}$%
. We then extend $F^{0}$ on $\mathbf{L}^{2}\left( \mathbf{R}^{N},\mathbf{R}%
^{N}\right) \times \mathcal{B}\left( \mathbf{R}^{N}\right) $ defining%
\begin{equation}
F^{0}\left( u,B\right) =\underset{D\in \mathcal{D},\overline{D}\subset 
\overset{o}{B}}{\sup }\left( F^{0}\right) ^{s}\left( u,D\right) =\underset{%
D\in \mathcal{D},\overline{D}\subset \overset{o}{B}}{\sup }\left(
F^{0}\right) ^{i}\left( u,D\right) .  \label{31}
\end{equation}

We define the family $\mathcal{R}\left( F\right) $ of Borel subsets of $%
\mathbf{R}^{N}$ through%
\[
\mathcal{R}\left( F\right) =\left\{ 
\begin{array}{r}
B\in \mathcal{B}\left( \mathbf{R}^{N}\right) \mid \forall u\in \mathbf{L}%
^{2}\left( \mathbf{R}^{N},\mathbf{R}^{N}\right) :\left( F^{0}\right)
_{+}^{s}\left( u,B\right) = \\ 
\underset{D\in \mathcal{D},\overline{D}\subset \overset{o}{B}}{\sup }\left(
F^{0}\right) ^{s}\left( u,D\right) =\underset{D\in \mathcal{D},\overline{B}%
\subset \overset{o}{D}}{\inf }\left( F^{0}\right) ^{s}\left( u,D\right) \\ 
=\left( F^{0}\right) _{-}^{s}\left( u,B\right)%
\end{array}%
\right\} .
\]

Then we prove (see \cite[Proposition 14.14]{Dal1}) that $\mathcal{R}\left(
F^{0}\right) $ is a rich family in $\mathcal{B}\left( \mathbf{R}^{N}\right) $
and $F^{0}=\left( F^{0}\right) ^{s}=\left( F^{0}\right) _{+}^{s}=\left(
F^{0}\right) _{-}^{s}=\left( F^{0}\right) _{+}^{i}=\left( F^{0}\right)
_{-}^{i}=\left( F^{0}\right) ^{i}$ on $\mathcal{R}\left( F^{0}\right) $. One
obtains, for every $u\in \mathbf{V}_{\Gamma _{1}}\left( \Omega \right) $ and
every $B\in \mathcal{R}\left( F^{0}\right) $%
\[
\begin{array}{ccl}
F^{0}\left( u,B\right) & = & \underset{z_{\varepsilon _{k}}\overset{\tau }{%
\rightharpoonup }0}{\inf }\underset{k\rightarrow +\infty }{\lim \sup }\left(
\Phi ^{\varepsilon _{k}}\left( z_{\varepsilon _{k}}\right) +F^{\varepsilon
_{k}}\left( u+z_{\varepsilon _{k}},B\right) \right) \\ 
& = & \underset{z_{\varepsilon _{k}}\overset{\tau }{\rightharpoonup }0}{\inf 
}\underset{k\rightarrow +\infty }{\lim \inf }\left( \Phi ^{\varepsilon
_{k}}\left( z_{\varepsilon _{k}}\right) +F^{\varepsilon _{k}}\left(
u+z_{\varepsilon _{k}},B\right) \right) .%
\end{array}%
\]

Let now $\varepsilon ^{\prime }$ denote any subsequence of $\varepsilon $.
Thanks to the above method, there exist a subsequence $\left( \varepsilon
_{k}^{\prime }\right) _{k}$, a functional $\mathcal{F}^{0}$ and a rich
family $\mathcal{R}\left( \mathcal{F}^{0}\right) $ such that, for every $%
u\in \mathbf{V}_{\Gamma _{1}}\left( \Omega \right) $ and every $B\in 
\mathcal{R}\left( \mathcal{F}^{0}\right) $%
\[
\begin{array}{ccl}
\mathcal{F}^{0}\left( u,B\right) & = & \underset{z_{\varepsilon _{k}^{\prime
}}\overset{\tau }{\rightharpoonup }0}{\inf }\underset{k\rightarrow +\infty }{%
\lim \sup }\left( \Phi ^{\varepsilon _{k}^{\prime }}\left( z_{\varepsilon
_{k}^{\prime }}\right) +F^{\varepsilon _{k}^{\prime }}\left(
u+z_{\varepsilon _{k}^{\prime }},B\right) \right) \\ 
& = & \underset{z_{\varepsilon _{k}^{\prime }}\overset{\tau }{%
\rightharpoonup }0}{\inf }\underset{k\rightarrow +\infty }{\lim \inf }\left(
\Phi ^{\varepsilon _{k}^{\prime }}\left( z_{\varepsilon _{k}^{\prime
}}\right) +F^{\varepsilon _{k}^{\prime }}\left( u+z_{\varepsilon
_{k}^{\prime }},B\right) \right) .%
\end{array}%
\]

Because $\mathcal{R}\left( F^{0}\right) \cap \mathcal{R}\left( \mathcal{F}%
^{0}\right) $ is still a rich family, one has%
\[
\forall u\in \mathbf{V}_{\Gamma _{1}}\left( \Omega \right) \text{, }\forall
B\in \mathcal{R}:F^{0}\left( u,.\right) =\mathcal{F}^{0}\left( u,.\right) 
\text{, on }\mathcal{R}\left( F^{0}\right) \cap \mathcal{R}\left( \mathcal{F}%
^{0}\right) .
\]

Because the countable intersection of rich families is a rich family too,
one can repeat\ the above reasoning and deduce the existence of a rich
family $\mathcal{R}$ in $\mathcal{B}\left( \mathbf{R}^{N}\right) $ on which
the above limits coincide. One thus obtains, for every $u\in \mathbf{V}%
_{\Gamma _{1}}\left( \Omega \right) $ and every $B\in \mathcal{R}$%
\begin{equation}
\left( \underset{\varepsilon \rightarrow 0}{\Gamma \text{-}\lim }%
G^{\varepsilon }\right) \left( u,\omega \right) =\Phi ^{0}\left( u\right)
+F^{0}\left( u,B\right) ,  \label{14}
\end{equation}%
where the $\Gamma $-limit is taken with respect to the topology $\tau $.

Thanks to the above properties 1., 2., 3. and 4. and to the relations (\ref%
{30}) and (\ref{31}), $F^{0}$ belongs to $\mathbb{F}$. Because $\Phi
^{\varepsilon }$ and $F^{\varepsilon }$ are quadratic, thanks to Corollary %
\ref{deux} and to Remark \ref{un}, there exist $\lambda \in \mathcal{M}_{0}$
finite, a symmetric matrix $\left( a_{ij}\right) _{i,j=1,..,N}$ of Borel
functions from $\mathbf{R}^{N}$ to $\mathbf{R}$ with $a_{ij}\left( x\right)
\zeta _{i}\zeta _{j}\geq 0$, $\forall \zeta \in \mathbf{R}^{N}$ and for q.e. 
$x\in \mathbf{R}^{N}$, such that, for every $u\in \mathbf{V}_{\Gamma
_{1}}\left( \Omega \right) $ and every $\omega \in \mathcal{R}\cap \mathcal{O%
}\left( \mathbf{R}^{N}\right) $%
\[
F^{0}\left( u,\omega \right) =\dint\nolimits_{\omega }u_{i}u_{j}d\mu _{ij},
\]%
with $\mu ^{\bullet }=\left( \mu _{ij}\right) _{i,j=1,..,N}=\left(
a_{ij}\lambda \right) _{i,j=1,..,N}+\infty _{\mathbf{R}^{N}\backslash
\Lambda }Id$, where $\Lambda $ is defined as in Remark \ref{un}.

Let us now precise the support of $\mu ^{\bullet }$. For every $u,v\in 
\mathbf{H}_{\Gamma _{1}}^{1}\left( \mathbf{R}^{N},\limfunc{div}\right) $,
such that $v_{\mid \Omega }=u_{\mid \Omega }$, one has%
\[
F^{0}\left( u,\mathbf{R}^{N}\right) =\dint\nolimits_{\mathbf{R}%
^{N}}v_{i}v_{j}d\mu _{ij},
\]%
because $F^{0}$ is local ($\mathbf{R}^{N}$ belongs to $\mathcal{R}$ because
every rich family is dense, and every dense family contains $\mathbf{R}^{N}$%
). One deduces that $supp\left( \mu ^{\bullet }\right) \subset \Omega \cup
\Gamma _{2}$. Thanks to (\ref{14}), one has%
\begin{equation}
0\leq \dint\nolimits_{\mathbf{R}^{N}}u_{i}u_{j}d\mu _{ij}+\Phi ^{0}\left(
u\right) \leq \underset{\varepsilon \rightarrow 0}{\lim \inf }\left( \Phi
^{\varepsilon }\left( u\right) +F^{\varepsilon }\left( u,\mathbf{R}%
^{N}\right) \right) .  \label{37}
\end{equation}

Taking $u\in \mathbf{H}_{0}^{1}\left( \Omega ,\limfunc{div}\right) =\left\{
u\in \mathbf{H}_{0}^{1}\left( \Omega ,\mathbf{R}^{N}\right) \mid \limfunc{div%
}\left( u\right) =0\right\} $, then, for every $\varepsilon >0$, $%
F^{\varepsilon }\left( u,\mathbf{R}^{N}\right) =0$, and $\lim
\inf_{\varepsilon \rightarrow 0}\Phi ^{\varepsilon }\left( u\right) =\Phi
^{0}\left( u\right) $. One deduces, using (\ref{37}), that $\int_{\Omega
}u_{i}u_{j}d\mu _{ij}=0$, and thus that $supp\left( \mu ^{\bullet }\right)
\subset \Gamma _{2}$, which ends the proof.
\end{proof}

\begin{remark}
\begin{enumerate}
\item We thus get Navier's wall law at the zeroth-order limit of the problem
(\ref{Peps}).

\item Theorem \ref{quatre} can be extended to every kind of obstacle
functional in $\mathbb{F}$, using Theorem \ref{theun} for the integral
representation. One can define, for example, sequences of obstacle
functionals on $\mathbf{H}^{1}\left( \mathbf{R}^{N},\mathbf{R}^{N}\right)
\times \mathcal{O}\left( \mathbf{R}^{N}\right) $ of the kind%
\[
\left( F^{\varepsilon }\right) ^{+}\left( u,\omega \right) =\left\{ 
\begin{array}{ll}
0 & \text{if }\widetilde{u}\geq 0\text{ q.e. on }\Gamma _{2,\varepsilon
}\cap \omega , \\ 
+\infty & \text{otherwise,}%
\end{array}%
\right.
\]%
the limit $\left( F^{0}\right) ^{+}$ of which is defined on $\mathbf{V}%
_{\Gamma _{1}}\left( \Omega \right) \times \left( \mathcal{R}^{+}\cap 
\mathcal{O}\left( \mathbf{R}^{N}\right) \right) $ (for some rich family $%
\mathcal{R}^{+}$) as%
\[
\left( F^{0}\right) ^{+}\left( u,\omega \right) =\dint\nolimits_{\omega \cap
\Gamma _{2}}u_{i}^{+}u_{j}^{+}d\mu _{ij},
\]%
where $u_{i}^{+}=\max \left( 0,u_{i}\right) $, $i=1,\ldots ,N$.

\item One proves that $\mu _{ij}\in \mathbf{H}^{-1/2}\left( \Gamma
_{2}\right) $, $\forall i,j=1,\ldots ,N$, where $\mu _{ij}$ is the measure
defined in Theorem \ref{quatre}. One first observes that the measure $%
\lambda $ defined in Theorem \ref{theun} belongs to $\mathbf{H}^{-1/2}\left(
\Gamma _{2}\right) ^{+}$. $\lambda $ is indeed finite.\ Because for every
compact subset $K\subset \Gamma _{2}$, one has $\lambda \left( K\right)
<+\infty $, hence $\lambda $ is a Radon nonnegative measure.\ Moreover,
because $\lambda $ is absolutely continuous with respect to the capacity $%
Cap $, we deduce from \cite[Theorem 2.2]{Dal2}, the existence of a Radon
measure $\varkappa \in \mathbf{H}^{-1/2}\left( \Gamma _{2}\right) $ and of a
Borel function $f:\Gamma _{2}\rightarrow \left[ 0,+\infty \right[ $ such
that $f=\frac{d\lambda }{d\varkappa }$.
\end{enumerate}
\end{remark}

Let us come back to the study of problem (\ref{Peps}). The solution $%
u^{\varepsilon }$ of (\ref{Peps}), with the homogeneous Dirichlet boundary
conditions on $\partial \Omega _{\varepsilon }$ is also the solution of the
minimization problem%
\begin{equation}
\underset{v\in \mathbf{L}^{2}\left( \mathbf{R}^{N},\mathbf{R}^{N}\right) }{%
\inf }\left( G^{\varepsilon }\left( v,\mathbf{R}^{N}\right)
+2\dint\nolimits_{\Omega _{\varepsilon }}\left( u^{\varepsilon }\cdot \nabla
\right) u^{\varepsilon }\cdot vdx-2\dint\nolimits_{\Omega _{\varepsilon
}}f\cdot vdx\right) .  \label{mineps}
\end{equation}

From Theorem \ref{quatre}, one deduces the following asymptotic behaviour of
the solution of (\ref{Peps}).

\begin{corollary}
The solution $\left( u^{\varepsilon },p^{\varepsilon }\right) $ of (\ref%
{Peps}), is such that $\left( u^{\varepsilon }\right) _{\varepsilon }$
converges to $u^{0}$ in the topology $\tau $ and $\left( \left(
p^{\varepsilon }\right) _{\mid \Omega }\right) _{\varepsilon }$ converges to 
$p^{0}$ in the strong topology of $\mathbf{L}^{2}\left( \Omega \right) /%
\mathbf{R}$, where $\left( u^{0},p^{0}\right) $ belongs to $\mathbf{V}%
_{0,\Gamma _{1}}\left( \Omega \right) \times \mathbf{L}^{2}\left( \Omega
\right) /\mathbf{R}$ and is the solution of the limit minimization problem%
\begin{equation}
\underset{v\in \mathbf{L}^{2}\left( \mathbf{R}^{N},\mathbf{R}^{N}\right) }{%
\inf }\left( G^{0}\left( v,\mathbf{R}^{N}\right) +2\dint\nolimits_{\Omega
}\left( u^{0}\cdot \nabla \right) u^{0}\cdot vdx-2\dint\nolimits_{\Omega
}f\cdot vdx\right) ,  \label{min0}
\end{equation}%
or of the limit problem with Navier law%
\begin{equation}
\left\{ 
\begin{array}{rlll}
-\nu \Delta u^{0}+\left( u^{0}\cdot \nabla \right) u^{0}+\nabla p^{0} & = & f
& \text{in }\Omega , \\ 
\limfunc{div}\left( u^{0}\right) & = & 0 & \text{in }\Omega , \\ 
u^{0} & = & 0 & \text{on }\Gamma _{1}, \\ 
u^{0}\cdot n & = & 0 & \text{on }\Gamma _{2}, \\ 
\left( I-n\otimes n\right) \nu \dfrac{\partial u^{0}}{\partial n}+\mu
^{\bullet }u^{0} & = & 0 & \text{on }\Gamma _{2}.%
\end{array}%
\right.  \label{P0}
\end{equation}
\end{corollary}

\begin{proof}
We first observe that, for every sequence $\left( v_{\varepsilon }\right)
_{\varepsilon }$ converging to $v$ in the topology $\tau $%
\[
\underset{\varepsilon \rightarrow 0}{\lim }\dint\nolimits_{\Omega
_{\varepsilon }}f\cdot v_{\varepsilon }dx=\dint\nolimits_{\Omega }f\cdot vdx,
\]

Thanks to the properties of the $\Gamma $-convergence, $\left(
u^{\varepsilon }\right) _{\varepsilon }$ converges to $u^{0}$ in the
topology $\tau $, with $u^{0}\in \mathbf{V}_{\Gamma _{1}}\left( \Omega
\right) $, and%
\[
\underset{\varepsilon \rightarrow 0}{\lim }G^{\varepsilon }\left(
u^{\varepsilon },\mathbf{R}^{N}\right) =G^{0}\left( u^{0},\mathbf{R}%
^{N}\right) =\nu \dint\nolimits_{\Omega }\left\vert \nabla u^{0}\right\vert
^{2}dx+\dint\nolimits_{\Gamma _{2}}\left( u^{0}\right) _{i}\left(
u^{0}\right) _{j}d\mu _{ij}.
\]

Then%
\[
\underset{\varepsilon \rightarrow 0}{\lim }\dint\nolimits_{\Omega
_{\varepsilon }}\left( u^{\varepsilon }\cdot \nabla \right) u^{\varepsilon
}\cdot v_{\varepsilon }dx=\dint\nolimits_{\Omega }\left( u^{0}\cdot \nabla
\right) u^{0}\cdot vdx,
\]%
for every sequence $\left( v_{\varepsilon }\right) _{\varepsilon }$
converging to $v$ in the topology $\tau $. For every $\varphi \in \mathbf{C}%
^{1}\left( \mathbf{R}^{N}\right) $, one has%
\[
\left\vert \dint\nolimits_{\Sigma _{\varepsilon }}u^{\varepsilon }\cdot
\nabla \varphi dx\right\vert \leq \left( \dint\nolimits_{\Sigma
_{\varepsilon }}\left\vert \nabla \varphi \right\vert ^{2}dx\right)
^{1/2}\left( \dint\nolimits_{\mathbf{R}^{N}}\left\vert u^{\varepsilon
}\right\vert ^{2}dx\right) ^{1/2},
\]%
and thus $\lim_{\varepsilon \rightarrow 0}\int_{\Sigma _{\varepsilon
}}u^{\varepsilon }\cdot \nabla \varphi dx=0$. Because $\limfunc{div}\left(
u^{\varepsilon }\right) =\limfunc{div}\left( u^{0}\right) =0$, and $%
u^{\varepsilon }=0$, q.e. on $\Gamma _{2}$, one has%
\[
0=\dint\nolimits_{\Omega _{\varepsilon }}u^{\varepsilon }\cdot \nabla
\varphi dx=\dint\nolimits_{\Omega }u^{\varepsilon }\cdot \nabla \varphi
dx+\dint\nolimits_{\Sigma _{\varepsilon }}u^{\varepsilon }\cdot \nabla
\varphi dx.
\]

Taking the limit of this equality, we obtain%
\[
0=\dint\nolimits_{\Omega }u^{0}\cdot \nabla \varphi
dx=\dint\nolimits_{\Gamma _{2}}u^{0}\cdot n\varphi d\Gamma _{2},
\]%
which proves that $u^{0}\cdot n=0$ on $\Gamma _{2}$. Thus $u^{0}\in \mathbf{V%
}_{0,\Gamma _{1}}\left( \Omega \right) $ is the solution of the problem (\ref%
{min0}). The variational formulation of (\ref{min0}) can be written as%
\[
\begin{array}{l}
\forall \varphi \in \mathbf{V}_{0,\Gamma _{1}}\left( \Omega \right)
:\dint_{\Omega }\left( -\nu \Delta u^{0}+\left( u^{0}\cdot \nabla \right)
u^{0}\right) \cdot \varphi dx \\ 
\qquad +\dint_{\Gamma _{2}}\nu \dfrac{\partial u^{0}}{\partial n}\cdot
\varphi d\Gamma _{2}+\dint_{\Gamma _{2}}\left( u^{0}\right) _{i}\varphi
_{j}d\mu _{ij}=\dint_{\Omega }f\cdot \varphi dx.%
\end{array}%
\]

There exists $p_{0}\in L^{2}\left( \Omega \right) /\mathbf{R}$ such that $%
-\nu \Delta u^{0}+\left( u^{0}\cdot \nabla \right) u^{0}-f=-\nabla p_{0}$.\
Thanks to Proposition \ref{estim}, the sequence $\left( \left(
p^{\varepsilon }\right) _{\mid \Omega }\right) _{\varepsilon }$ converges to 
$p^{0}$ in the strong topology of $\mathbf{L}^{2}\left( \Omega \right) /%
\mathbf{R}$. Because $\varphi \cdot n=0$ on $\Gamma _{2}$, with $n=\left(
0,0,1\right) $, one has: $\nu \frac{\partial u^{0}}{\partial n}\cdot \varphi
=\left( Id-n\otimes n\right) \nu \frac{\partial u^{0}}{\partial n}\cdot
\varphi $, which ends the proof.
\end{proof}

\section{Special cases}

We intend to specialize the general result obtained in Theorem \ref{quatre},
in two cases where the boundary $\Gamma _{2,\varepsilon }$ can be defined
through some Lipschitz continuous function.

\subsection{Periodic case}

In this section, we suppose that $\Omega \subset \left\{ x_{3}>0\right\} $
with $\partial \Omega \cap \left\{ x_{3}=0\right\} =\Gamma _{2}$, $\Gamma
_{2}$ containing $0$. We define $Y=\left( -1/2,1/2\right) ^{2}$ and consider
a $Y$-periodic function $h\in \mathbf{C}_{c}^{2}\left( Y,\mathbf{R}%
_{+}\right) $. For every $k\in \mathbf{Z}^{2}$, we define $Y_{\varepsilon
}^{k}=\left( -\varepsilon /2,\varepsilon /2\right) ^{2}+\left(
k_{1}\varepsilon ,k_{2}\varepsilon \right) $, and let $I_{\varepsilon
}=\left\{ k\in \mathbf{Z}^{2}\mid Y_{\varepsilon }^{k}\subset \Gamma
_{2}\right\} $. We define $h_{\varepsilon }$ on $\Gamma _{2}$ through%
\[
h_{\varepsilon }\left( x^{\prime }\right) =\left\{ 
\begin{array}{ll}
h\left( \dfrac{x^{\prime }}{\varepsilon }\right) & \text{if there exists }%
k\in I_{\varepsilon }\text{ such that }x^{\prime }=\left( x_{1},x_{2}\right)
\in Y_{\varepsilon }^{k}, \\ 
0 & \text{otherwise}%
\end{array}%
\right.
\]%
and $\Sigma _{\varepsilon }$ through%
\[
\Sigma _{\varepsilon }=\left\{ x\in \mathbf{R}^{3}\mid x^{\prime }=\left(
x_{1},x_{2}\right) \in \Gamma _{2}\text{, }-\varepsilon h_{\varepsilon
}\left( x^{\prime }\right) <x_{3}<0\right\} .
\]

Thanks to Theorem \ref{quatre}, there exist a rich family $\mathcal{R}%
\subset \mathcal{B}\left( \mathbf{R}^{3}\right) $, a symmetric matrix $%
\left( \mu _{ij}\right) _{i,j=1,\ldots ,N}$ of Borel measures having the
same support contained in $\Gamma _{2}$, absolutely continuous with respect
to the capacity $Cap$, and satisfying $\mu _{ij}\left( B\right) \zeta
_{i}\zeta _{j}\geq 0$, $\forall \zeta \in \mathbf{R}^{3}$, $\forall B\in 
\mathcal{B}\left( \mathbf{R}^{3}\right) $, such that, for every $u\in 
\mathbf{V}_{\Gamma _{1}}\left( \Omega \right) $ and every $\omega \in 
\mathcal{R}\cap \mathcal{O}\left( \mathbf{R}^{3}\right) $%
\begin{equation}
\inf \left\{ 
\begin{array}{l}
\underset{\varepsilon \rightarrow 0}{\lim \inf }\Phi ^{\varepsilon }\left(
z_{\varepsilon }\right) \mid u+z_{\varepsilon }=0\text{ on} \\ 
\qquad \left\{ x_{3}=-\varepsilon h_{\varepsilon }\left( x^{\prime }\right)
\right\} \cap \omega \text{ and }z_{\varepsilon }\overset{\tau }{\underset{%
\varepsilon \rightarrow 0}{\rightharpoonup }}0%
\end{array}%
\right\} =\dint\nolimits_{\omega \cap \Gamma _{2}}u_{i}u_{j}d\mu _{ij},
\label{44}
\end{equation}%
where $\Phi ^{\varepsilon }$ is the energy functional defined in (\ref%
{Phieps}).

Because the lower boundary $\Gamma _{2,\varepsilon }$ of $\Sigma
_{\varepsilon }$, defined through the equality $\Gamma _{2,\varepsilon
}=\left\{ \left( x^{\prime },x_{3}\right) \mid x_{3}=-\varepsilon
h_{\varepsilon }\left( x^{\prime }\right) \right\} $, has a periodic
structure, the measures $\mu _{ij}$, $i,j=1,\ldots ,N$, are invariant under
translations on $\Gamma _{2}$. This implies $\mu _{ij}=K_{ij}dx^{\prime }$,
where $K_{ij}$, $i,j=1,2,3$, are constants in $\overline{\mathbf{R}}$
satisfying $K_{ij}\zeta _{i}\zeta _{j}\geq 0$, $\forall \zeta \in \mathbf{R}%
^{3}$.

The purpose of this section is to identify these constants $K_{ij}$, $%
i,j=1,2,3$. We observe that we do not have to determine $K_{i3}$, $i=1,2,3$,
because, in the limit problem, one has $u\cdot n=u\cdot e_{3}=u_{3}=0$.

\begin{theorem}
The limit Navier wall law of the limit problem (\ref{P0}) is in this case%
\[
\dfrac{\partial \left( u^{0}\right) _{m}}{\partial x_{3}}=c_{m}\left(
u^{0}\right) _{m}\text{, on }\Gamma _{2}\text{, }m=1,2,
\]%
where the constants $c_{m}$ are defined in (\ref{cm}).
\end{theorem}

\begin{proof}
We define the set $Z_{h}=\left\{ x\mid x^{\prime }\in Y\text{, }-h\left(
x^{\prime }\right) <x_{3}<0\right\} $ and consider in $Z_{h}$ the local
Stokes problems for $m=1,2$%
\begin{equation}
\left\{ 
\begin{array}{rlll}
-\Delta w^{m}+\nabla q^{m} & = & e^{m} & \text{in }Z_{h}, \\ 
\limfunc{div}\left( w^{m}\right) & = & 0 & \text{in }Z_{h}, \\ 
w^{m} & = & e^{m} & \text{on }\left\{ x_{3}=-h\left( x^{\prime }\right)
\right\} , \\ 
w^{m} & = & 0 & \text{on }\left\{ x_{3}=0\right\} , \\ 
w^{m},q^{m} &  &  & Y\text{-periodic,}%
\end{array}%
\right.  \label{Pm}
\end{equation}%
where $e^{m}$ is the $m$-th vector of the canonical basis of $\mathbf{R}^{3}$%
. Lax-Milgram' Theorem implies that (\ref{Pm}) has a unique solution $\left(
w^{m},q^{m}\right) $ with%
\[
\begin{array}{rll}
w^{m} & \in & \mathbf{V}\left( Z_{h}\right) =\left\{ 
\begin{array}{l}
u\in \mathbf{H}^{1}\left( Z_{h},\mathbf{R}^{3}\right) \mid \limfunc{div}%
\left( u\right) =0\text{ in }Z_{h}\text{,} \\ 
\qquad u=0\text{ on }\left\{ x_{3}=0\right\} \text{, }u\text{ }Y\text{%
-periodic}%
\end{array}%
\right\} \\ 
q^{m} & \in & \mathbf{L}^{2}\left( Z_{h}\right) /\mathbf{R}\text{, }q^{m}%
\text{ }Y\text{-periodic.}%
\end{array}%
\]

Let $z_{h}=\max_{x^{\prime }\in Y}h\left( x^{\prime }\right) $ and choose $%
H>z_{h}$. We define%
\[
\widetilde{Z}_{h}=\left\{ x\mid x^{\prime }\in Y\text{, }-H<x_{3}<-h\left(
x^{\prime }\right) \right\}
\]%
and consider in $\widetilde{Z}_{h}$ problems similar to (\ref{Pm}) except
that we impose $\widetilde{w}^{m}=e^{m}$ on $\left\{ x_{3}=-h\left(
x^{\prime }\right) \right\} $ and $\widetilde{w}^{m}=0$ on $\left\{
x_{3}=-H\right\} $. Let us define%
\[
\begin{array}{rll}
\widetilde{\Sigma }_{\varepsilon } & = & \left\{ x\in \mathbf{R}^{3}\mid
x^{\prime }=\left( x_{1},x_{2}\right) \in \Gamma _{2}\text{, }-\varepsilon
H<x_{3}<-\varepsilon h_{\varepsilon }\left( x^{\prime }\right) \right\} , \\ 
B_{\varepsilon } & = & \left\{ x\in \mathbf{R}^{3}\mid x^{\prime }=\left(
x_{1},x_{2}\right) \in \Gamma _{2}\text{, }-\varepsilon H<x_{3}<0\right\}%
\end{array}%
\]%
and the functions $\left( w^{\varepsilon m},q^{\varepsilon m}\right) $ and $%
\left( \widetilde{w}^{\varepsilon m},\widetilde{q}^{\varepsilon m}\right) $
through%
\[
\left\{ 
\begin{array}{rllrll}
w^{\varepsilon m}\left( x\right) & = & w^{m}\left( \dfrac{x}{\varepsilon }%
\right) , & q^{\varepsilon m}\left( x\right) & = & q^{m}\left( \dfrac{x}{%
\varepsilon }\right) , \\ 
\widetilde{w}^{\varepsilon m}\left( x\right) & = & \widetilde{w}^{m}\left( 
\dfrac{x}{\varepsilon }\right) , & \widetilde{q}^{\varepsilon m}\left(
x\right) & = & \widetilde{q}^{m}\left( \dfrac{x}{\varepsilon }\right) .%
\end{array}%
\right.
\]

We finally build the function $z_{\varepsilon }^{0m}$, on $B_{\varepsilon }$%
, through%
\[
z_{\varepsilon }^{0m}\left( x\right) =\left\{ 
\begin{array}{ll}
w^{\varepsilon m}\left( x\right) & \text{if }x\in \Sigma _{\varepsilon }, \\ 
e^{m} & \text{on }\left\{ x_{3}=-\varepsilon h_{\varepsilon }\left(
x^{\prime }\right) \right\} , \\ 
\widetilde{w}^{\varepsilon m}\left( x\right) & \text{on }\widetilde{\Sigma }%
_{\varepsilon }.%
\end{array}%
\right.
\]

Because $h=0$ on $\partial Y$, one can suppose that $z_{\varepsilon }^{0m}=0$
on $\partial \Gamma _{2}\times \left( -\varepsilon H,0\right) $. This
implies that $z_{\varepsilon }^{0m}\in \mathbf{H}_{\Gamma _{1}}^{1}\left( 
\mathbf{R}^{3},\limfunc{div}\right) $ and $z_{\varepsilon }^{0m}=0$ on $%
\partial B_{\varepsilon }$. Moreover%
\[
\begin{array}{lll}
\underset{\varepsilon \rightarrow 0}{\lim }\dint_{\mathbf{R}^{N}}\left\vert
z_{\varepsilon }^{0m}\right\vert ^{2}dx & = & \underset{\varepsilon
\rightarrow 0}{\lim }\dint_{B_{\varepsilon }}\left\vert z_{\varepsilon
}^{0m}\right\vert ^{2}dx \\ 
& = & \underset{\varepsilon \rightarrow 0}{\lim }\underset{k\in
I_{\varepsilon }}{\dsum }\dint_{Y_{\varepsilon }^{k}}\dint_{-\varepsilon
H}^{0}\left\vert z_{\varepsilon }^{0m}\right\vert ^{2}dx \\ 
& = & \underset{\varepsilon \rightarrow 0}{\lim }\left( 
\begin{array}{l}
\underset{k\in I_{\varepsilon }}{\dsum \varepsilon ^{3}}\dint_{Y}%
\dint_{-H}^{-h\left( x^{\prime }\right) }\left\vert \widetilde{w}^{m}\left(
x\right) \right\vert ^{2}dx \\ 
\qquad +\underset{k\in I_{\varepsilon }}{\dsum \varepsilon ^{3}}%
\dint_{Y}\dint_{-h\left( x^{\prime }\right) }^{0}\left\vert w^{m}\left(
x\right) \right\vert ^{2}dx%
\end{array}%
\right) \\ 
& = & 0%
\end{array}%
\]%
and%
\[
\begin{array}{lll}
\underset{\varepsilon \rightarrow 0}{\lim }\Phi ^{\varepsilon }\left(
z_{\varepsilon }^{0m}\right) = & = & \underset{\varepsilon \rightarrow 0}{%
\lim }\nu \varepsilon \dint_{\Sigma _{\varepsilon }}\left\vert \nabla
z_{\varepsilon }^{0m}\right\vert ^{2}dx \\ 
& = & \underset{k\in I_{\varepsilon }}{\nu \dsum \varepsilon ^{2}}%
\dint_{Y}\dint_{-h\left( x^{\prime }\right) }^{0}\left\vert \nabla
w^{m}\left( x\right) \right\vert ^{2}dx \\ 
& = & \nu \left\vert \Gamma _{2}\right\vert c_{m},%
\end{array}%
\]%
with%
\begin{equation}
c_{m}=\dint\nolimits_{Z_{h}}\left\vert \nabla w^{m}\right\vert ^{2}dx.
\label{cm}
\end{equation}

Taking $u=-e^{m}$ on $\Sigma _{\varepsilon }$, in (\ref{44}), one obtains%
\[
\begin{array}{lll}
K_{mm}\left\vert \Gamma _{2}\right\vert & = & \inf \left\{ 
\begin{array}{l}
\underset{\varepsilon \rightarrow 0}{\lim \inf }\Phi ^{\varepsilon }\left(
z_{\varepsilon }\right) \mid z_{\varepsilon }=e^{m}\text{ on }\left\{
x_{3}=-\varepsilon h_{\varepsilon }\left( x^{\prime }\right) \right\} , \\ 
\qquad z_{\varepsilon }\overset{\tau }{\underset{\varepsilon \rightarrow 0}{%
\rightharpoonup }}0%
\end{array}%
\right\} \\ 
& \leq & \underset{\varepsilon \rightarrow 0}{\lim }\Phi ^{\varepsilon
}\left( z_{\varepsilon }^{0m}\right) =\nu c_{m}\left\vert \Gamma
_{2}\right\vert .%
\end{array}%
\]

This implies%
\begin{equation}
K_{mm}\left\vert \Gamma _{2}\right\vert \leq \nu c_{m}\left\vert \Gamma
_{2}\right\vert .  \label{ineq1}
\end{equation}

Take any sequence $\left( z_{\varepsilon }\right) _{\varepsilon }\subset 
\mathbf{H}_{\Gamma _{1}}^{1}\left( \mathbf{R}^{3},\limfunc{div}\right) $
such that $z_{\varepsilon }=e^{m}$ on the surface $\left\{
x_{3}=-\varepsilon h_{\varepsilon }\left( x^{\prime }\right) \right\} $ and $%
\left( z_{\varepsilon }\right) _{\varepsilon }$ converges to $0$ in the
topology $\tau $.\ We write the subdifferential inequality%
\begin{equation}
\Phi ^{\varepsilon }\left( z_{\varepsilon }\right) \geq \Phi ^{\varepsilon
}\left( z_{\varepsilon }^{0m}\right) +2\nu \varepsilon
\dint\nolimits_{\Sigma _{\varepsilon }}\nabla z_{\varepsilon }^{0m}\cdot
\nabla \left( z_{\varepsilon }-z_{\varepsilon }^{0m}\right) dx.
\label{subdiff}
\end{equation}

We observe that%
\[
\begin{array}{ccl}
\varepsilon \dint_{\Sigma _{\varepsilon }}\nabla z_{\varepsilon }^{0m}\cdot
\nabla \left( z_{\varepsilon }-z_{\varepsilon }^{0m}\right) dx & = & 
-\varepsilon \dint_{\Sigma _{\varepsilon }}\Delta z_{\varepsilon }^{0m}\cdot
\left( z_{\varepsilon }-z_{\varepsilon }^{0m}\right) dx \\ 
&  & \qquad -\varepsilon \dint_{\Gamma _{2}}\dfrac{\partial z_{\varepsilon
}^{0m}}{\partial n}\cdot \left( z_{\varepsilon }-z_{\varepsilon
}^{0m}\right) d\Gamma _{2}.%
\end{array}%
\]

Using the regularity (at least $\mathbf{H}^{2}$) of $w^{m}$, we obtain%
\[
\varepsilon \Delta z_{\varepsilon }^{0m}\underset{\varepsilon \rightarrow 0}{%
\rightharpoonup }\mathbf{1}_{\Gamma _{2}}\dint\nolimits_{Z_{h}}\Delta
w^{m}\left( x\right) dx,
\]%
where the convergence takes place in the weak topology of $\mathbf{L}%
^{2}\left( \mathbf{R}^{3},\mathbf{R}^{3}\right) $ and $\mathbf{1}_{\Gamma
_{2}}$ is the characteristic function of $\Gamma _{2}$.\ Then%
\[
\begin{array}{l}
\left\vert \varepsilon \dint\nolimits_{\Gamma _{2}}\dfrac{\partial
z_{\varepsilon }^{0m}}{\partial n}\cdot \left( z_{\varepsilon
}-z_{\varepsilon }^{0m}\right) d\Gamma _{2}\right\vert \\ 
\qquad \leq \left( \dint\nolimits_{\Gamma _{2}}\left\vert \dfrac{\partial
w^{m}}{\partial n}\right\vert ^{2}d\Gamma _{2}\right) ^{1/2}\left(
\dint\nolimits_{\mathbf{R}^{3}}\left\vert z_{\varepsilon }-z_{\varepsilon
}^{0m}\right\vert ^{2}dx\right) ^{1/2}.%
\end{array}%
\]

Because $\left( z_{\varepsilon }-z_{0}^{\varepsilon m}\right) _{\varepsilon
} $ converges to $0$ in the strong topology $\mathbf{L}^{2}\left( \mathbf{R}%
^{3},\mathbf{R}^{3}\right) $, we have%
\[
\underset{\varepsilon \rightarrow 0}{\lim }\varepsilon
\dint\nolimits_{\Sigma _{\varepsilon }}\nabla z_{\varepsilon }^{0m}\cdot
\nabla \left( z_{\varepsilon }-z_{\varepsilon }^{0m}\right) dx=0.
\]

Taking the $\lim \inf $ in (\ref{subdiff}), one obtains%
\[
\underset{\varepsilon \rightarrow 0}{\lim \inf }\Phi ^{\varepsilon }\left(
z_{\varepsilon }\right) \geq \underset{\varepsilon \rightarrow 0}{\lim \inf }%
\Phi ^{\varepsilon }\left( z_{\varepsilon }^{0m}\right) =\nu c_{m}\left\vert
\Gamma _{2}\right\vert .
\]

In this last inequality, taking the infimum with respect to all sequences $%
\left( z_{\varepsilon }\right) _{\varepsilon }$ satisfying the imposed
conditions, one obtains: $K_{mm}\left\vert \Gamma _{2}\right\vert \geq \nu
c_{m}\left\vert \Gamma _{2}\right\vert $. This inequality and (\ref{ineq1})
imply: $K_{mm}=\nu c_{m}$. Taking now $u=-\left( e^{1}+e^{2}\right) $ on $%
\Sigma _{\varepsilon }$ in (\ref{44}), one obtains%
\[
\begin{array}{lll}
\left( K_{11}+2K_{12}+K_{22}\right) \left\vert \Gamma _{2}\right\vert & = & 
\inf \left\{ 
\begin{array}{l}
\underset{\varepsilon \rightarrow 0}{\lim \inf }\Phi ^{\varepsilon }\left(
z_{\varepsilon }\right) \mid z_{\varepsilon }=e^{1}+e^{2}\text{ on} \\ 
\qquad \left\{ x_{3}=-\varepsilon h_{\varepsilon }\left( x^{\prime }\right)
\right\} \text{, }z_{\varepsilon }\overset{\tau }{\underset{\varepsilon
\rightarrow 0}{\rightharpoonup }}0%
\end{array}%
\right\} \\ 
& \leq & \underset{\varepsilon \rightarrow 0}{\lim }\Phi ^{\varepsilon
}\left( z_{\varepsilon }^{01}+z_{\varepsilon }^{02}\right) .%
\end{array}%
\]

Because $\int_{Z_{h}}\nabla w^{1}\cdot \nabla w^{2}dz=0$, we have%
\[
\lim_{\varepsilon \rightarrow 0}\Phi ^{\varepsilon }\left( z_{\varepsilon
}^{01}+z_{\varepsilon }^{02}\right) =\nu \left\vert \Gamma _{2}\right\vert
\left( c_{1}+c_{2}\right) .
\]

This implies: $K_{12}\leq 0$, through the above expression of $K_{mm}$.
Writing a subdifferential inequality as in (\ref{subdiff}), one obtains: $%
K_{12}\geq 0$, which implies: $K_{12}=0$.
\end{proof}

\subsection{Case where $h_{\protect\varepsilon }$ is independent of $\protect%
\varepsilon $}

As in the previous section, we still suppose that $\Omega \subset \left\{
x_{3}>0\right\} $ and $\partial \Omega \cap \left\{ x_{3}=0\right\} =\Gamma
_{2}$. But, we here suppose that the boundary $\Gamma _{2,\varepsilon }$ is
given as%
\[
\Gamma _{2,\varepsilon }=\left\{ \left( x^{\prime },x_{3}\right) \mid
x_{3}=-\varepsilon h\left( x^{\prime }\right) \right\}
\]%
where $h$ is a Lipschitz continuous function satisfying $h\left( x^{\prime
}\right) >0$, $\forall x^{\prime }\in \Gamma _{2}$. We have the following
result.

\begin{theorem}
\label{indep}Under the preceding hypothesis, the Navier wall law is in this
case%
\[
\left( Id-n\otimes n\right) \dfrac{\partial u^{0}}{\partial n}+\dfrac{u^{0}}{%
h}=0\text{, on }\Gamma _{2}.
\]
\end{theorem}

\begin{proof}
Thanks to Theorem \ref{quatre}, there exist a rich family $\mathcal{R}%
_{\Gamma _{2}}\subset \mathcal{B}\left( \Sigma \right) $, a symmetric matrix 
$\left( \mu _{ij}\right) _{i,j=1,\ldots ,N}$ of Borel measures having their
support contained in $\Gamma _{2}$, which are absolutely continuous with
respect to the capacity $Cap$, and satisfying $\mu _{ij}\left( B\right)
\zeta _{i}\zeta _{j}\geq 0$, $\forall \zeta \in \mathbf{R}^{3}$, $\forall
B\in \mathcal{B}\left( \Sigma \right) $, such that, for every $u\in \mathbf{V%
}_{\Gamma _{1}}\left( \Omega \right) $ and every $\omega \in \mathcal{R}%
_{\Gamma _{2}}\cap \mathcal{O}\left( \Gamma _{2}\right) $%
\begin{equation}
\dint\nolimits_{\omega }u_{i}u_{j}d\mu _{ij}=\inf \left\{ 
\begin{array}{l}
\underset{\varepsilon \rightarrow 0}{\lim \inf }\Phi ^{\varepsilon }\left(
z_{\varepsilon }\right) \mid u+z_{\varepsilon }=0\text{ on } \\ 
\qquad \left\{ x_{3}=-\varepsilon h\left( x^{\prime }\right) \right\} \cap
\omega \text{, }z_{\varepsilon }\overset{\tau }{\underset{\varepsilon
\rightarrow 0}{\rightharpoonup }}0%
\end{array}%
\right\} .  \label{hind}
\end{equation}

Take $u=-e^{1}$ on $\left\{ x_{3}=-\varepsilon h\left( x^{\prime }\right)
\right\} $.\ Then choose $\omega \in \mathcal{R}_{\Gamma _{2}}\cap \mathcal{O%
}\left( \Gamma _{2}\right) $, an open subset $\omega ^{\varepsilon }$ of $%
\mathbf{R}^{2}$ such that $\omega ^{\varepsilon }\setminus \overline{\omega }%
=\left\{ x^{\prime }\in \mathbf{R}^{2}\mid 0<d\left( x^{\prime },\partial
\omega \right) <\varepsilon \right\} $ and $\varphi ^{\varepsilon }\in 
\mathbf{C}^{1}\left( \mathbf{R}^{2}\right) $ with $0\leq \varphi
^{\varepsilon }\leq 1$ such that%
\[
\left\{ 
\begin{array}{rlll}
\varphi ^{\varepsilon } & = & 1 & \text{in }\omega , \\ 
\varphi ^{\varepsilon } & = & 0 & \text{on }\partial \omega _{\varepsilon }.%
\end{array}%
\right.
\]

We define the function $w^{1\varepsilon }$ through%
\[
\left\{ 
\begin{array}{lll}
\left( w^{1\varepsilon }\right) _{1}\left( x\right) & = & \dfrac{x_{3}}{%
\varepsilon h\left( x^{\prime }\right) }\varphi ^{\varepsilon }\left(
x^{\prime }\right) , \\ 
\left( w^{1\varepsilon }\right) _{2}\left( x\right) & = & 0, \\ 
\left( w^{1\varepsilon }\right) _{3}\left( x\right) & = & \dfrac{\varepsilon 
}{2}\left( \dfrac{\partial h}{\partial x_{1}}\left( x^{\prime }\right)
\varphi ^{\varepsilon }\left( x^{\prime }\right) -\dfrac{\partial \varphi
^{\varepsilon }}{\partial x_{1}}\left( x^{\prime }\right) h\left( x^{\prime
}\right) \right) \\ 
&  & \qquad +\dfrac{\left( x_{3}\right) ^{2}}{2}\left( \dfrac{1}{\varepsilon
h\left( x^{\prime }\right) }\dfrac{\partial \varphi ^{\varepsilon }}{%
\partial x_{1}}\left( x^{\prime }\right) -\dfrac{\varphi ^{\varepsilon
}\left( x^{\prime }\right) }{\varepsilon h^{2}\left( x^{\prime }\right) }%
\dfrac{\partial h}{\partial x_{1}}\left( x^{\prime }\right) \right) .%
\end{array}%
\right.
\]

One has $\limfunc{div}\left( w^{1\varepsilon }\right) =0$, $\forall
\varepsilon >0$, and $w^{1\varepsilon }=e^{1}$ on $\left\{
x_{3}=-\varepsilon h\left( x^{\prime }\right) \right\} \cap \left( \omega
\times \left( -\infty ,0\right) \right) $. We now consider the problem%
\begin{equation}
\left\{ 
\begin{array}{rlll}
-\Delta \zeta ^{1\varepsilon }+\nabla \varpi ^{1\varepsilon } & = & e^{1} & 
\text{in }\Omega , \\ 
\limfunc{div}\left( \zeta ^{1\varepsilon }\right) & = & 0 & \text{in }\Omega
, \\ 
\zeta ^{1\varepsilon } & = & 0 & \text{on }\Gamma _{1}, \\ 
\zeta ^{1\varepsilon } & = & \left( 0,0,\dfrac{1}{2}\left( \dfrac{\partial h%
}{\partial x_{1}}\left( x^{\prime }\right) \varphi ^{\varepsilon }\left(
x^{\prime }\right) -\dfrac{\partial \varphi _{\varepsilon }}{\partial x_{1}}%
\left( x^{\prime }\right) h\left( x^{\prime }\right) \right) \right) & \text{%
on }\Gamma _{2}.%
\end{array}%
\right.  \label{Ph1eps}
\end{equation}

The problem (\ref{Ph1eps}) has a unique solution $\left( \zeta
^{1\varepsilon },\varpi ^{1\varepsilon }\right) \in \mathbf{H}_{\Gamma
_{1}}^{1}\left( \Omega ,\limfunc{div}\right) \times \mathbf{L}^{2}\left(
\Omega \right) /\mathbf{R}$, satisfying%
\[
\dint\nolimits_{\Omega }\left\vert \nabla \zeta ^{1\varepsilon }\right\vert
^{2}dx\leq C\text{ ; }\dint\nolimits_{\Omega }\left\vert \zeta
^{1\varepsilon }\right\vert ^{2}dx\leq C,
\]%
where $C$ is a constant independent of $\varepsilon $. Let $H>z_{h}$, with $%
z_{h}=\max_{\Gamma _{2}}h$. We define the function $\widetilde{w}%
^{1\varepsilon }$ in $D_{\varepsilon }=\left\{ x\mid -H<x_{3}<-\varepsilon
h\left( x^{\prime }\right) \right\} $ through%
\[
\left\{ 
\begin{array}{rll}
\left( \widetilde{w}^{1\varepsilon }\right) _{1}\left( x\right) & = & \dfrac{%
x_{3}+H}{\varepsilon \left( H-h\left( x^{\prime }\right) \right) }\varphi
^{\varepsilon }\left( x^{\prime }\right) , \\ 
\left( \widetilde{w}^{1\varepsilon }\right) _{2}\left( x\right) & = & 0, \\ 
\left( \widetilde{w}^{1\varepsilon }\right) _{3}\left( x\right) & = & \dfrac{%
\varepsilon }{2}\left( \dfrac{\partial h}{\partial x_{1}}\left( x^{\prime
}\right) \varphi ^{\varepsilon }\left( x^{\prime }\right) -\dfrac{\partial
\varphi ^{\varepsilon }}{\partial x_{1}}\left( x^{\prime }\right) \left(
H-h\left( x^{\prime }\right) \right) \right) \\ 
&  & \qquad -\dfrac{\left( x_{3}+H\right) ^{2}}{2}\left( 
\begin{array}{l}
\dfrac{1}{\varepsilon \left( H-h\left( x^{\prime }\right) \right) }\dfrac{%
\partial \varphi ^{\varepsilon }}{\partial x_{1}}\left( x^{\prime }\right)
\\ 
+\dfrac{\varphi ^{\varepsilon }\left( x^{\prime }\right) }{\varepsilon
\left( H-h\left( x^{\prime }\right) \right) ^{2}}\dfrac{\partial h}{\partial
x_{1}}\left( x^{\prime }\right)%
\end{array}%
\right) .%
\end{array}%
\right.
\]

We consider the bounded, smooth and open subset $\Omega _{H}=\left\{ x\mid
x_{3}>-H\right\} $ and $\partial \Omega _{H}\cap \left\{ x\mid
x_{3}=-H\right\} =\Gamma _{2}$, and the solution $\left( \zeta
_{H}^{1\varepsilon },\omega _{H}^{1\varepsilon }\right) $ of the problem%
\[
\left\{ 
\begin{array}{rlll}
-\Delta \zeta _{H}^{1\varepsilon }+\nabla \varpi _{H}^{1\varepsilon } & = & 
e^{1} & \text{in }\Omega _{H}, \\ 
\limfunc{div}\left( \zeta _{H}^{1\varepsilon }\right) & = & 0 & \text{in }%
\Omega _{H}, \\ 
\zeta _{H}^{1\varepsilon } & = & 0 & \text{on }\Omega _{H}\setminus \Gamma
_{2}, \\ 
\zeta _{H}^{1\varepsilon } & = & \left( 0,0,\dfrac{1}{2}\left( 
\begin{array}{l}
\dfrac{\partial h}{\partial x_{1}}\left( x^{\prime }\right) \varphi
^{\varepsilon }\left( x^{\prime }\right) \\ 
-\dfrac{\partial \varphi ^{\varepsilon }}{\partial x_{1}}\left( x^{\prime
}\right) \left( H-h\left( x^{\prime }\right) \right)%
\end{array}%
\right) \right) & \text{on }\Gamma _{2}.%
\end{array}%
\right.
\]

Let us define the function $z_{0}^{1,\varepsilon }$ through%
\[
z_{\varepsilon }^{0,1}=\left\{ 
\begin{array}{ll}
\varepsilon \zeta ^{1\varepsilon } & \text{in }\Omega , \\ 
w^{1\varepsilon } & \text{in }\Sigma _{\varepsilon }, \\ 
\widetilde{w}^{1\varepsilon } & \text{in }D_{\varepsilon }, \\ 
\varepsilon \zeta _{H}^{1\varepsilon } & \text{in }\Omega _{H}.%
\end{array}%
\text{ }\right.
\]

One immediately verifies that $z_{\varepsilon }^{0,1}\in \mathbf{H}_{\Gamma
_{1}}^{1}\left( \mathbf{R}^{3},\limfunc{div}\right) $, $z_{\varepsilon
}^{0,1}=e^{1}$ on the surface $\left\{ x_{3}=-\varepsilon h\left( x^{\prime
}\right) \right\} \cap \left( \omega \times \left( -\infty ,0\right) \right) 
$, $\left( z_{\varepsilon }^{0,1}\right) _{\varepsilon }$ converges to $0$
in the strong topology of $\mathbf{L}^{2}\left( \mathbf{R}^{3},\mathbf{R}%
^{3}\right) $ and%
\[
\underset{\varepsilon \rightarrow 0}{\lim }\Phi ^{\varepsilon }\left(
z_{\varepsilon }^{0,1}\right) =\underset{\varepsilon \rightarrow 0}{\lim }%
\nu \varepsilon \dint\nolimits_{\omega ^{\varepsilon }\times \left(
-\varepsilon h\left( x^{\prime }\right) ,0\right) }\left\vert \nabla
z_{\varepsilon }^{0,1}\right\vert ^{2}dx=\nu \dint\nolimits_{\omega }\dfrac{%
dx^{\prime }}{h\left( x^{\prime }\right) }.
\]

One thus deduces from (\ref{hind}) within this context%
\[
\mu _{11}\left( \omega \right) \leq \nu \dint\nolimits_{\omega }\dfrac{%
dx^{\prime }}{h\left( x^{\prime }\right) }.
\]

Furthermore, taking $\left( z_{\varepsilon }\right) _{\varepsilon }\subset 
\mathbf{H}_{\Gamma _{1}}^{1}\left( \mathbf{R}^{3},\limfunc{div}\right) $, $%
z_{\varepsilon }=e^{1}$ on $\left\{ x_{3}=-\varepsilon h\left( x^{\prime
}\right) \right\} \cap \left( \omega \times \left( -\infty ,0\right) \right) 
$, $\left( z_{\varepsilon }\right) _{\varepsilon }$ converges to $0$ in the
topology $\tau $, and using the subdifferential inequality%
\[
\begin{array}{l}
\Phi ^{\varepsilon }\left( z_{\varepsilon }\right) \geq \Phi ^{\varepsilon
}\left( z_{\varepsilon }^{0,1}\right) \\ 
\qquad +\nu \varepsilon \dint\nolimits_{\Sigma _{\varepsilon }}\nabla
z_{\varepsilon }^{0,1}\cdot \nabla \left( z_{\varepsilon }-z_{\varepsilon
}^{0,1}\right) dx+\nu \dint\nolimits_{\Omega }\nabla z_{\varepsilon
}^{0,1}\cdot \nabla \left( z_{\varepsilon }-z_{\varepsilon }^{0,1}\right) dx,%
\end{array}%
\]%
we prove that $\mu _{11}\left( \omega \right) \geq \nu \int_{\omega
}dx^{\prime }/h\left( x^{\prime }\right) $. This implies the equality: $\mu
_{11}\left( \omega \right) =\nu \int_{\omega }dx^{\prime }/h\left( x^{\prime
}\right) $ and, since this equality is true for every $\omega \in \mathcal{R}%
_{\Gamma _{2}}\cap \mathcal{O}\left( \Gamma _{2}\right) $, we obtain $\mu
_{11}=\nu dx^{\prime }/h\left( x^{\prime }\right) $.

Choosing now $u=-e^{2\text{ }}$ on $\Sigma _{\varepsilon }$, we can build a
test-function $z_{\varepsilon }^{0,2}$ in a similar way and prove: $\mu
_{22}=\nu dx^{\prime }/h\left( x^{\prime }\right) $.

Finally, taking $u=-\left( e^{1}+e^{2}\right) $ on $\Sigma _{\varepsilon }$,
we consider the sequence $\left( z_{\varepsilon }^{0}\right) _{\varepsilon }$
defined through: $z_{\varepsilon }^{0}=z_{\varepsilon }^{0,1}+z_{\varepsilon
}^{0,2}$. One deduces from the above computations that%
\[
\underset{\varepsilon \rightarrow 0}{\lim }\Phi ^{\varepsilon }\left(
z_{\varepsilon }^{0}\right) =\underset{\varepsilon \rightarrow 0}{\lim }\Phi
^{\varepsilon }\left( z_{\varepsilon }^{0,1}+z_{\varepsilon }^{0,2}\right)
=2\nu \dint\nolimits_{\omega }\dfrac{dx^{\prime }}{h\left( x^{\prime
}\right) }
\]%
and, as in the periodic case, that $\mu _{12}=0$. The boundary conditions on 
$\Gamma _{2}$ can thus be written as%
\[
\left\{ 
\begin{array}{rll}
\left( u^{0}\right) _{3} & = & 0, \\ 
\dfrac{\partial \left( u^{0}\right) _{m}}{\partial x_{3}} & = & \dfrac{1}{h}%
\left( u^{0}\right) _{m}\text{, }m=1,2,%
\end{array}%
\right.
\]%
which ends the proof.
\end{proof}

\begin{remark}
In a general way, if $\Sigma _{\varepsilon }=\left\{ \sigma +tn\mid \sigma
\in \Gamma _{2}\text{, }-\varepsilon h\left( \sigma \right) <t<0\right\} $,
with $h$ positive and Lipschitz continuous on $\Gamma _{2}$, we can prove
that the limit law is%
\[
\left\{ 
\begin{array}{rll}
\left( Id-n\otimes n\right) \dfrac{\partial u^{0}}{\partial n}+\dfrac{u^{0}}{%
h} & = & 0, \\ 
u^{0}\cdot n & = & 0.%
\end{array}%
\right.
\]
\end{remark}

\section{Optimal control problem}

For a given real $m>0$, we consider the set $\Xi _{m}$ of all matrices $%
\mathbf{h}=Diag\left( h_{i}\right) _{i=1,..,N}$ of functions $h_{i}:\Gamma
_{2}\rightarrow \left[ 0,+\infty \right] $, $d\Gamma _{2}$-measurable and
such that%
\[
\dint\nolimits_{\Gamma _{2}}h_{i}d\Gamma _{2}=m\text{, }\forall i=1,\ldots
,N.
\]

We suppose that $\partial \Omega $ is $\mathbf{C}^{2}$ and consider the
Navier-Stokes problem, with Navier wall law, according to Theorem \ref{indep}%
\begin{equation}
\left\{ 
\begin{array}{rrrl}
-\nu \Delta u^{h}+\left( u^{h}\cdot \nabla \right) u^{h}+\nabla p^{h} & = & f
& \text{in }\Omega , \\ 
\limfunc{div}\left( u^{h}\right) & = & 0 & \text{in }\Omega , \\ 
\mathbf{h}\left( Id-n\otimes n\right) \dfrac{\partial u^{h}}{\partial n}%
+u^{h} & = & 0 & \text{on }\Gamma _{2}, \\ 
u^{h}\cdot n & = & 0 & \text{on }\Gamma _{2}, \\ 
u^{h} & = & 0 & \text{on }\Gamma _{1},%
\end{array}%
\right.  \label{stokesm}
\end{equation}%
which has a unique solution $\left( u^{h},p^{h}\right) \in \mathbf{V}%
_{0,\Gamma _{1}}\left( \Omega \right) \times \mathbf{L}^{2}\left( \Omega
\right) /\mathbf{R}$. We define the functional $\mathbf{F}$ defined on $\Xi
_{m}\times \mathbf{H}_{\Gamma _{1}}^{1}\left( \Omega ,\func{div}\right) $
and associated to (\ref{stokesm}) through%
\[
\mathbf{F}\left( \mathbf{h},u\right) =\left\{ 
\begin{array}{ll}
\dfrac{\nu }{2}\dint_{\Omega }\left\vert \nabla u\right\vert ^{2}dx+\dfrac{1%
}{2}\underset{i=1}{\overset{N}{\dsum }}\dint_{\Gamma _{2}}\dfrac{\left(
u_{i}\right) ^{2}}{h_{i}}d\Gamma _{2} &  \\ 
\qquad +\dint_{\Omega }\left( u^{h}\cdot \nabla \right) u^{h}\cdot
udx-\dint_{\Omega }f\cdot udx & \text{if }u\in \mathbf{V}_{0,\Gamma
_{1}}\left( \Omega \right) , \\ 
+\infty & \text{otherwise.}%
\end{array}%
\right.
\]

We consider the optimal control problem (\ref{Popm}), which means that the
cost functional is here taken as the global energy. We observe that%
\[
\mathbf{F}\left( \mathbf{h},u^{h}\right) =-\dint\nolimits_{\Omega }f\cdot
u^{h}dx.
\]

This implies that the minimization of $\mathbf{F}$, with respect to $u$ on
the set $\mathbf{V}_{0,\Gamma _{1}}\left( \Omega \right) $, is equivalent to
the maximization of the work of the external forces on this set. The problem
(\ref{Popm}) has a unique minimizer when Poincar\'{e}'s inequality%
\[
\left( \dint\nolimits_{\Gamma _{2}}\left\vert u_{i}\right\vert d\Gamma
_{2}\right) ^{2}\leq \dint\nolimits_{\Gamma _{2}}h_{i}d\Gamma
_{2}\dint\nolimits_{\Gamma _{2}}\frac{\left( u_{i}\right) ^{2}}{h_{i}}%
d\Gamma _{2},
\]%
becomes an equality, for every $i=1,\ldots ,N$, that is when%
\[
h_{i}^{m}=m\frac{\left\vert u_{i}^{m}\right\vert _{\Gamma _{2}}}{%
\dint\nolimits_{\Gamma _{2}}\left\vert u_{i}^{m}\right\vert d\Gamma _{2}},
\]%
where $\left( u^{m},p^{m}\right) $ is the solution of%
\[
\left\{ 
\begin{array}{rlll}
-\nu \Delta u^{m}+\left( u^{m}\cdot \nabla \right) u^{m}+\nabla p^{m} & = & f
& \text{in }\Omega , \\ 
\limfunc{div}\left( u^{m}\right) & = & 0 & \text{in }\Omega , \\ 
u^{m}\cdot n & = & 0 & \text{on }\Gamma _{2}, \\ 
u^{m} & = & 0 & \text{on }\Gamma _{1}, \\ 
\left( Id-n\otimes n\right) \dfrac{\partial u^{m}}{\partial n}\qquad &  &  & 
\\ 
+\dfrac{1}{m}\left( 
\begin{array}{c}
sign\left( \left( u^{m}\right) _{1}\left( x\right) \right) \dint_{\Gamma
_{2}}\left\vert \left( u^{m}\right) _{1}\right\vert d\Gamma _{2} \\ 
\vdots \\ 
sign\left( \left( u^{m}\right) _{N}\left( x\right) \right) \dint_{\Gamma
_{2}}\left\vert \left( u^{m}\right) _{N}\right\vert d\Gamma _{2}%
\end{array}%
\right) & = & 0 & \text{on }\Gamma _{2}.%
\end{array}%
\right.
\]

Trivially, the study of the $\Gamma $-convergence of the sequence of the
energies associated to (\ref{Popm}), when $m$ goes to $0$ and relatively to
the weak topology of $\mathbf{H}^{1}\left( \Omega ,\mathbf{R}^{N}\right) $,
will lead to the following conclusions: $\left( u^{m}\right) _{m}$ converges
to $u^{0}$ in the weak topology of $\mathbf{H}^{1}\left( \Omega ,\mathbf{R}%
^{N}\right) $, $\left( p^{m}\right) _{m}$ converges to $p^{0}$ in the strong
topology of $\mathbf{L}^{2}\left( \Omega \right) /\mathbf{R}$, where $\left(
u^{0},p^{0}\right) $ is the solution of the problem%
\begin{equation}
\left\{ 
\begin{array}{rlll}
-\nu \Delta u^{0}+\left( u^{0}\cdot \nabla \right) u^{0}+\nabla p^{0} & = & f
& \text{in }\Omega , \\ 
\limfunc{div}\left( u^{0}\right) & = & 0 & \text{in }\Omega , \\ 
u^{0} & = & 0 & \text{on }\partial \Omega .%
\end{array}%
\right.  \label{P0o}
\end{equation}

In order to study the asymptotic behavior of $\left( \left( u^{m}/m\right)
_{\mid \Gamma _{2}}\right) _{m}$, we introduce the following linearized
perturbation of the Navier-Stokes problem (\ref{P0o})%
\begin{equation}
\left\{ 
\begin{array}{rlll}
-\nu \Delta u^{0,m}+\nabla p^{0,m} & = & f-\left( u^{m}\cdot \nabla \right)
u^{m} & \text{in }\Omega , \\ 
\limfunc{div}\left( u^{0,m}\right) & = & 0 & \text{in }\Omega , \\ 
u^{0,m} & = & 0 & \text{on }\partial \Omega .%
\end{array}%
\right.  \label{P0m}
\end{equation}

The problem (\ref{P0m}) is a Stokes system, the source term of which is $%
f-\left( u^{m}\cdot \nabla \right) u^{m}$. Consider now the functional $%
I_{m} $ defined on $\mathbf{V}_{0,\Gamma _{1}}\left( \Omega \right) $ through%
\[
\begin{array}{ccl}
I_{m}\left( v\right) & = & \dfrac{m\nu }{2}\dint_{\Omega }\left\vert \nabla
v\right\vert ^{2}dx+\dfrac{1}{2}\underset{i=1}{\overset{N}{\dsum }}\left(
\dint_{\Gamma _{2}}\left\vert v_{i}\right\vert d\Gamma _{2}\right) ^{2} \\ 
&  & \qquad +\dint_{\Gamma _{2}}\left( Id-n\otimes n\right) \dfrac{\partial
u^{0,m}}{\partial n}\cdot vd\Gamma _{2}.%
\end{array}%
\]

$I_{m}$ has a unique minimizer $\left( v^{m},q^{m}\right) \in \mathbf{V}%
_{0,\Gamma _{1}}\left( \Omega \right) \times \mathbf{L}^{2}\left( \Omega
\right) /\mathbf{R}$ which is the solution of the problem%
\[
\left\{ 
\begin{array}{rlll}
-\nu m\Delta v^{m}+\nabla q^{m} & = & 0 & \text{in }\Omega , \\ 
\limfunc{div}\left( v^{m}\right) & = & 0 & \text{in }\Omega , \\ 
v^{m}\cdot n & = & 0 & \text{on }\Gamma _{2}, \\ 
v^{m} & = & 0 & \text{on }\Gamma _{1}, \\ 
\left( Id-n\otimes n\right) \dfrac{\partial u^{0,m}}{\partial n}+m\left(
Id-n\otimes n\right) \dfrac{\partial v^{m}}{\partial n}\qquad &  &  &  \\ 
+\left( 
\begin{array}{c}
sign\left( \left( v^{m}\right) _{1}\right) \dint_{\Gamma _{2}}\left\vert
\left( v^{m}\right) _{1}\right\vert d\Gamma _{2} \\ 
\vdots \\ 
sign\left( \left( v^{m}\right) _{N}\right) \dint_{\Gamma _{2}}\left\vert
\left( v^{m}\right) _{N}\right\vert d\Gamma _{2}%
\end{array}%
\right) & = & 0 & \text{on }\Gamma _{2}.%
\end{array}%
\right.
\]

We observe that the couple $\left( v^{m},q^{m}\right) $ defined through%
\[
v^{m}=\dfrac{u^{m}-u^{0,m}}{m}\text{ ; }q^{m}=p^{m}-p^{0,m},
\]%
is the minimizer of $I_{m}$. For every $\varphi \in \mathbf{H}^{1/2}\left(
\Gamma _{2},\mathbf{R}^{N}\right) $, there exists a unique extension $%
v_{\varphi }\in \mathbf{V}_{0,\Gamma _{1}}\left( \Omega \right) $ of $%
\varphi $ defined through%
\[
\dint\nolimits_{\Omega }\left\vert \nabla v_{\varphi }\right\vert ^{2}dx=%
\underset{\left\{ w\in \mathbf{V}_{0,\Gamma _{1}}\left( \Omega \right) \mid
w\mid _{\Gamma _{2}}=\varphi \right\} }{\inf }\dint\nolimits_{\Omega
}\left\vert \nabla w\right\vert ^{2}dx.
\]

Let us denote $\mathcal{M}\left( \Gamma _{2},\mathbf{R}^{N}\right) $ the
space of finite Radon measures on $\Gamma _{2}$ with values in $\mathbf{R}%
^{N}$. We consider the functional $J_{m}$ defined on $\mathcal{M}\left(
\Gamma _{2},\mathbf{R}^{N}\right) $ through%
\[
J_{m}\left( \varphi \right) =\left\{ 
\begin{array}{ll}
\dfrac{m\nu }{2}\dint_{\Omega }\left\vert \nabla v_{\varphi }\right\vert
^{2}dx+\dfrac{1}{2}\underset{i=1}{\overset{N}{\dsum }}\left( \dint_{\Gamma
_{2}}\left\vert \varphi _{i}\right\vert d\Gamma _{2}\right) ^{2} &  \\ 
\qquad +\dint_{\Gamma _{2}}\left( Id-n\otimes n\right) \dfrac{\partial
u^{0,m}}{\partial n}\cdot \varphi d\Gamma _{2} & \text{if }\varphi \in 
\mathbf{H}^{1/2}\left( \Gamma _{2},\mathbf{R}^{N}\right) \\ 
& \text{and }\varphi \cdot n=0\text{ on }\Gamma _{2}, \\ 
+\infty & \text{otherwise.}%
\end{array}%
\right.
\]

Then $\left( v^{m}\right) _{\mid \Gamma _{2}}$ is the unique minimizer of $%
J_{m}$.

\begin{proposition}
\label{six}One has the following properties.

\begin{enumerate}
\item $\sup_{m}\sum_{i=1}^{n}\left( \int_{\Gamma _{2}}\left\vert
v_{i}^{m}\right\vert d\Gamma _{2}\right) <+\infty $.

\item The sequence $\left( J_{m}\right) _{m}$ $\Gamma $-converges, when $m$
tends to $0$ and with respect to the weak$^{\ast }$ topology of $\mathcal{M}%
\left( \Gamma _{2},\mathbf{R}^{N}\right) $, to the functional $J$ defined
from $\mathcal{M}\left( \Gamma _{2},\mathbf{R}^{N}\right) $ to $\mathbf{R}$
through%
\[
J\left( \lambda \right) =\underset{i=1}{\overset{N}{\dsum }}\left(
\left\vert \lambda _{i}\right\vert \left( \Gamma _{2}\right) \right)
^{2}+\dint\nolimits_{\Gamma _{2}}\left( Id-n\otimes n\right) \dfrac{\partial
u^{0}}{\partial n}d\lambda ,
\]%
where $\left\vert \lambda _{i}\right\vert \left( \Gamma _{2}\right) $ is the
total variation of $\lambda _{i}$ on $\Gamma _{2}$.
\end{enumerate}
\end{proposition}

\begin{proof}
1. Remark that a regularity property of the boundary $\partial \Omega $
implies that%
\[
\underset{m}{\sup }\left\Vert \left( Id-n\otimes n\right) \dfrac{\partial
u^{0,m}}{\partial n}\right\Vert _{\mathbf{L}^{\infty }\left( \Gamma _{2},%
\mathbf{R}^{N}\right) }<+\infty .
\]

One thus obtains%
\[
J_{m}\left( \left( v^{m}\right) _{\mid \Gamma _{2}}\right) \geq \dfrac{1}{2}%
\underset{i=1}{\overset{N}{\dsum }}\left( \dint\nolimits_{\Gamma
_{2}}\left\vert v_{i}^{m}\right\vert d\Gamma _{2}\right) ^{2}-\dfrac{C}{2}%
\underset{i=1}{\overset{N}{\dsum }}\left( \dint\nolimits_{\Gamma
_{2}}\left\vert v_{i}^{m}\right\vert d\Gamma _{2}\right) .
\]

Moreover%
\[
\sup_{m}J_{m}\left( \left( v^{m}\right) _{\mid \Gamma _{2}}\right) \leq
\sup_{m}J_{m}\left( 0\right) =0\Rightarrow \underset{m}{\sup }\underset{i=1}{%
\overset{N}{\dsum }}\left( \dint\nolimits_{\Gamma _{2}}\left\vert
v_{i}^{m}\right\vert d\Gamma _{2}\right) \leq C.
\]

This implies the existence of a subsequence of $\left( \left( v^{m}\right)
_{\mid \Gamma _{2}}\right) _{m}$, still denoted $\left( \left( v^{m}\right)
_{\mid \Gamma _{2}}\right) _{m}$, which converges to some $\lambda $ in the
weak$^{\ast }$ topology of $\mathcal{M}\left( \Gamma _{2},\mathbf{R}%
^{N}\right) $.

\noindent 2. Choose any sequence $\left( \varphi ^{m}\right) _{m}\subset 
\mathbf{H}^{1/2}\left( \Gamma _{2},\mathbf{R}^{N}\right) $, satisfying $%
\varphi ^{m}\cdot n=0$, on $\Gamma _{2}$ and converging to $\lambda $ in the
weak$^{\ast }$ topology of $\mathcal{M}\left( \Gamma _{2},\mathbf{R}%
^{N}\right) $. The functional $\mu \mapsto \left\vert \mu \right\vert $,
where $\left\vert \mu \right\vert $ is the total variation of $\mu $, being
lower semi-continuous on $\mathcal{M}\left( \Gamma _{2}\right) $, one has%
\[
\underset{m\rightarrow 0}{\lim \inf }\dint\nolimits_{\Gamma _{2}}\left\vert
\varphi _{i}^{m}\right\vert d\Gamma _{2}\geq \left\vert \lambda
_{i}\right\vert \left( \Gamma _{2}\right) .
\]

Thanks to the regularity of the boundary, $\left( \left( Id-n\otimes
n\right) \frac{\partial u^{0,m}}{\partial n}\right) _{m}$ uniformly
converges to $\left( Id-n\otimes n\right) \frac{\partial u^{0}}{\partial n}$%
, hence%
\[
\underset{m\rightarrow 0}{\lim \inf }\dint\nolimits_{\Gamma _{2}}\left(
Id-n\otimes n\right) \dfrac{\partial u^{0,m}}{\partial n}\cdot \varphi
^{m}d\Gamma _{2}\geq \dint\nolimits_{\Gamma _{2}}\left( \left( Id-n\otimes
n\right) \dfrac{\partial u^{0}}{\partial n}\right) _{i}d\lambda _{i}.
\]

This implies%
\begin{equation}
\underset{m\rightarrow 0}{\lim \inf }J_{m}\left( \varphi ^{m}\right) \geq
J\left( \lambda \right) .  \label{97}
\end{equation}

In order to prove the $\Gamma $-$\lim \sup $ property, let us suppose that $%
\Omega \subset \left\{ x_{N}<0\right\} $ and $\partial \Omega \cap \left\{
x_{N}=0\right\} =\Gamma _{2}$ (in fact using a system of local coordinates,
one can then study the case of every smooth surface $\Gamma _{2}$). We
define $x^{\prime }=\left( x_{1},\ldots ,x_{N-1}\right) $ and the
nonnegative and smooth function $\rho _{\varepsilon }$ through%
\[
\rho _{\varepsilon }\left( x^{\prime }\right) =\left\{ 
\begin{array}{ll}
\dfrac{C}{\varepsilon ^{N-1}}\exp \left( -\dfrac{\varepsilon ^{2}}{%
\varepsilon ^{2}-\left\vert x^{\prime }\right\vert ^{2}}\right) & \text{if }%
\left\vert x^{\prime }\right\vert <\varepsilon , \\ 
0 & \text{if }\left\vert x^{\prime }\right\vert \geq \varepsilon ,%
\end{array}%
\right.
\]%
where%
\[
C=\left( \dint\nolimits_{B_{N-1}\left( 0,1\right) }\exp \left( \frac{-1}{%
1-\left\vert \zeta \right\vert ^{2}}\right) d\zeta \right) ^{-1}.
\]

Let $\left( \omega _{\left[ 1/\varepsilon \right] }\right) _{\varepsilon }$,
where $\left[ 1/\varepsilon \right] $ denotes the entire part of $%
1/\varepsilon $, be a sequence of open subsets of $\Gamma _{2}$ such that%
\[
\left\{ 
\begin{array}{l}
\omega _{1}\subset \omega _{2}\subset \ldots \subset \omega _{\left[
1/\varepsilon \right] }\subset \ldots \subset \Gamma _{2}, \\ 
\underset{\varepsilon }{\cup }\omega _{\left[ 1/\varepsilon \right] }=\Gamma
_{2}, \\ 
d\left( \omega _{\left[ 1/\varepsilon \right] },\partial \Gamma _{2}\right)
=\varepsilon .%
\end{array}%
\right.
\]

We associate the partition of unity $\left( \eta _{\varepsilon }\right)
_{\varepsilon }$ through%
\[
\left\{ 
\begin{array}{l}
\eta _{\varepsilon }\in \mathbf{C}_{c}^{\infty }\left( \omega _{\left[
1/\varepsilon \right] }\right) , \\ 
\eta _{\varepsilon }\left( x^{\prime }\right) =1\text{ in }\omega _{\left[
1/\varepsilon \right] -1}\text{ (}\left[ 1/\varepsilon \right] -1=\left[
1/\varepsilon ^{\prime }\right] \text{, with }\varepsilon ^{\prime }=\dfrac{%
\varepsilon }{1-\varepsilon }\text{),} \\ 
0\leq \eta _{\varepsilon }\left( x^{\prime }\right) \leq 1\text{, }\forall
x^{\prime }\in \Gamma _{2}\text{, }\forall \varepsilon >0.%
\end{array}%
\right.
\]

For $\lambda =\left( \lambda _{1},\ldots ,\lambda _{N-1},0\right) \in 
\mathcal{M}\left( \Gamma _{2},\mathbf{R}^{N}\right) $, we define the
vectorial measure $\lambda ^{\varepsilon }$ through $\lambda ^{\varepsilon
}=\left( \lambda \ast \rho _{\varepsilon }\right) \eta _{\varepsilon }$.\ We
observe that $\lambda ^{\varepsilon }\in \mathbf{C}_{c}^{\infty }\left(
\Gamma _{2},\mathbf{R}^{N}\right) $ and%
\[
\left\{ 
\begin{array}{rlll}
\lambda ^{\varepsilon } & \underset{\varepsilon \rightarrow 0}{%
\rightharpoonup } & \lambda & w^{\ast }\text{-}\mathcal{M}\left( \Gamma _{2},%
\mathbf{R}^{N}\right) , \\ 
\left\vert \nabla \lambda ^{\varepsilon }\right\vert \left( x^{\prime
}\right) & \leq & \dfrac{C}{\varepsilon ^{N}} & \forall x^{\prime }\in
\Gamma _{2}.%
\end{array}%
\right.
\]

We build the function $w^{\varepsilon }$%
\[
\left\{ 
\begin{array}{rlll}
\left( w^{\varepsilon }\right) _{i}\left( x\right) & = & \dfrac{\varepsilon
-x_{N}}{\varepsilon }\left( \lambda ^{\varepsilon }\right) _{i}\left(
x^{\prime }\right) & i=1,\ldots ,N-1\text{, }\forall x\in \Omega , \\ 
\left( w^{\varepsilon }\right) _{N}\left( x\right) & = & \dfrac{\limfunc{div}%
\left( \lambda ^{\varepsilon }\left( x^{\prime }\right) \right) }{2}\left( 
\dfrac{\left( \varepsilon -x_{N}\right) ^{2}}{\varepsilon }-\varepsilon
\right) . & 
\end{array}%
\right.
\]

We immediately observe that $w^{\varepsilon }\in \mathbf{H}^{1}\left( \Omega
,\mathbf{R}^{N}\right) $ and%
\[
\left\{ 
\begin{array}{rlll}
\limfunc{div}\left( w^{\varepsilon }\right) & = & 0 & \text{in }\Omega , \\ 
\left( w^{\varepsilon }\right) _{N} & = & 0 & \text{on }\Gamma _{2}, \\ 
w^{\varepsilon } & = & 0 & \text{on }\Gamma _{1},%
\end{array}%
\right.
\]%
that is $w^{\varepsilon }\in \mathbf{V}_{0,\Gamma _{1}}\left( \Omega \right) 
$, for every $\varepsilon >0$. We now define%
\[
\left\{ 
\begin{array}{rll}
\varepsilon & = & m^{\frac{1}{4N}}, \\ 
w^{m} & = & w^{m^{\frac{1}{4N}}}, \\ 
\lambda ^{m} & = & \lambda ^{m^{\frac{1}{4N}}}.%
\end{array}%
\right.
\]

One has%
\[
\left\{ 
\begin{array}{rll}
m\dint_{\Omega }\left\vert \nabla w^{m}\right\vert ^{2}dx & \leq & C\sqrt{m},
\\ 
J_{m}\left( \lambda ^{m}\right) & = & I_{m}\left( v_{\lambda ^{m}}\right)
\leq I_{m}\left( w^{m}\right) ,%
\end{array}%
\right.
\]%
hence%
\[
\underset{m\rightarrow 0}{\lim \sup }J_{m}\left( \lambda ^{m}\right) \leq 
\underset{m\rightarrow 0}{\lim \sup }I_{m}\left( w^{m}\right) =J\left(
\lambda \right) .
\]

This inequality and (\ref{97}) end the proof.
\end{proof}

One has the following result.

\begin{theorem}
Let%
\[
\begin{array}{rll}
M_{i} & = & \underset{\sigma \in \Gamma _{2}}{\max }\left\vert \left( \left(
Id-n\otimes n\right) \dfrac{\partial u^{0}}{\partial n}\right) _{i}\left(
\sigma \right) \right\vert , \\ 
K_{i}^{\pm } & = & \left\{ \sigma \in \Gamma _{2}\mid \left( \left(
Id-n\otimes n\right) \dfrac{\partial u^{0}}{\partial n}\right) _{i}\left(
\sigma \right) =\pm M_{i}\right\} .%
\end{array}%
\]%
We have the following properties.

\begin{enumerate}
\item When $m$ goes to $0$, the sequence $\left( \left( u^{m}/m\right)
_{\mid \Gamma _{2}}\right) _{m}$ converges in the weak$^{\ast }$ topology of
the space $\mathcal{M}\left( \Gamma _{2},\mathbf{R}^{N}\right) $ to a
vectorial measure $\lambda =\left( \lambda _{i}\right) _{i=1,\ldots ,N}$
such that $supp\left( \lambda _{i}\right) \subseteq K_{i}^{+}\cup K_{i}^{-}$%
, with $\lambda _{i}$ positive on $K_{i}^{-}$ and negative on $K_{i}^{+}$, $%
i=1,\ldots ,N$.

\item $\int_{\Gamma _{2}}\left( \left( Id-n\otimes n\right) \frac{\partial
u^{0}}{\partial n}\right) _{i}d\lambda _{i}=-M_{i}$, $i=1,\ldots ,N$.

\item $\lim_{m\rightarrow 0}\int_{\Gamma _{2}}\left\vert
u_{i}^{m}/m\right\vert d\Gamma _{2}=\left\vert \lambda _{i}\right\vert
\left( \Gamma _{2}\right) =M_{i}$, $i=1,\ldots ,N$.

\item When $m$ goes to $0$, the sequence $\left( h_{i}^{m}/m\right) _{m}$
converges in the weak$^{\ast }$ topology of $\mathcal{M}\left( \Gamma _{2},%
\mathbf{R}^{N}\right) $ to a measure $\overline{\lambda }_{i}$ such that $%
supp\left( \overline{\lambda }_{i}\right) \subseteq K_{i}^{+}\cup K_{i}^{-}$%
, $\overline{\lambda }_{i}$ is positive on $K_{i}^{-}$ and negative on $%
K_{i}^{+}$, and $\left\vert \overline{\lambda }_{i}\right\vert \left( \Gamma
_{2}\right) =1$, $i=1,\ldots ,N$.
\end{enumerate}
\end{theorem}

\begin{proof}
One deduces from Proposition \ref{six} and from the properties of the $%
\Gamma $-convergence that $\left( \left( v^{m}\right) _{\mid \Gamma
_{2}}\right) _{m}=\left( \left( u^{m}/m\right) _{\mid \Gamma _{2}}\right)
_{m}$ converges in the weak$^{\ast }$ topology of $\mathcal{M}\left( \Gamma
_{2},\mathbf{R}^{N}\right) $, when $m$ goes to $0$, to a measure $\lambda
=\left( \lambda _{i}\right) _{i=1,\ldots ,N}$ such that $J\left( \lambda
\right) =\min_{\upsilon \in \mathcal{M}\left( \Gamma _{2},\mathbf{R}%
^{N}\right) }J\left( \upsilon \right) $. Define%
\[
\mathcal{M}_{1}\left( \Gamma _{2},\mathbf{R}^{N}\right) =\left\{ \mu \in 
\mathcal{M}\left( \Gamma _{2},\mathbf{R}^{N}\right) \mid \left\vert \mu
_{i}\right\vert \left( \Gamma _{2}\right) =1\text{, }i=1,\ldots ,N\right\}
\]%
and consider the functional $\widetilde{J}$ defined from $\left[ 0,+\infty %
\right[ ^{N}\times \mathcal{M}_{1}\left( \Gamma _{2},\mathbf{R}^{N}\right) $
to $\mathbf{R}$ through%
\[
\begin{array}{l}
\widetilde{J}\left( \left( t_{1},\ldots ,t_{N}\right) ,\left( \mu
_{1},\ldots ,\mu _{N}\right) \right) \\ 
\qquad 
\begin{array}{cl}
= & J\left( \left( t_{1}\mu _{1},\ldots ,t_{N}\mu _{N}\right) \right) \\ 
= & \dfrac{1}{2}\underset{i=1}{\overset{N}{\dsum }}\left( t_{i}\right) ^{2}+%
\underset{i=1}{\overset{N}{\dsum }}t_{i}\dint_{\Gamma _{2}}\left( \left(
Id-n\otimes n\right) \dfrac{\partial u^{0}}{\partial n}\right) _{i}d\mu _{i}.%
\end{array}%
\end{array}%
\]

One has%
\begin{equation}
\underset{\upsilon \in \mathcal{M}\left( \Gamma _{2},\mathbf{R}^{N}\right) }{%
\min }J\left( \upsilon \right) =\underset{\mu \in \mathcal{M}_{1}\left(
\Gamma _{2},\mathbf{R}^{N}\right) }{\min }\underset{\underset{i=1,..,N}{%
t_{i}\geq 0}}{\min }\widetilde{J}\left( \left( t_{1},\ldots ,t_{N}\right)
,\left( \mu _{1},\ldots ,\mu _{N}\right) \right) .  \label{105}
\end{equation}

The minimum of (\ref{105}) with respect to $t=\left( t_{1},\ldots
,t_{N}\right) $ exists if%
\[
\dint\nolimits_{\Gamma _{2}}\left( \left( Id-n\otimes n\right) \frac{%
\partial u^{0}}{\partial n}\right) _{i}d\mu _{i}\leq 0\text{, }\forall
i=1,\ldots ,N.
\]

Let us now find the minimum with respect to $\mu \in \mathcal{M}_{1}\left(
\Gamma _{2},\mathbf{R}^{N}\right) $. One has%
\[
-\dint\nolimits_{\Gamma _{2}}\left( \left( Id-n\otimes n\right) \dfrac{%
\partial u^{0}}{\partial n}\right) _{i}d\mu _{i}\geq -M_{i},
\]%
for every $\mu \in \mathcal{M}_{1}\left( \Gamma _{2},\mathbf{R}^{N}\right) $
such that%
\[
\dint\nolimits_{\Gamma _{2}}\left( \left( Id-n\otimes n\right) \frac{%
\partial u^{0}}{\partial n}\right) _{i}d\mu _{i}\leq 0\text{, }\forall
i=1,\ldots ,N,
\]%
the minimum being reached in the case of equality, that is if and only if $%
supp\left( \mu _{i}\right) \subset K_{i}^{+}\cup K_{i}^{-}$. One has $%
\lambda _{i}=M_{i}\mu _{i}$, $i=1,\ldots ,N$. Remarking that $\overline{%
\lambda }_{i}=\mu _{i}$, one observes that $\left( h_{i}^{m}/m\right) _{m}$
converges in the weak$^{\ast }$ topology of $\mathcal{M}\left( \Gamma _{2},%
\mathbf{R}^{N}\right) $, when $m$ tends to $0$, to $\overline{\lambda }_{i}$%
, and the same result occurs for the sequence $\left( \left( \left\vert
u_{i}^{m}\right\vert _{\Gamma _{2}}\right) /\int_{\Gamma _{2}}\left\vert
u_{i}^{m}\right\vert d\Gamma _{2}\right) _{m}$. The sequence $\left(
h_{i}^{m}/m\right) _{m}$ converges in $\mathcal{M}\left( \Gamma _{2},\mathbf{%
R}^{N}\right) $-weak$^{\ast }$ to a probability measure $\overline{\lambda }%
_{i}$ ($\overline{\lambda }_{i}\left( \Gamma _{2}\right) =1$) with support
in the set of points of $\Gamma _{2}$ where the shear motions, given through 
$\left( Id-n\otimes n\right) \frac{\partial u^{0}}{\partial n}$, are large
for the limit flow described through (\ref{P0o}).
\end{proof}

\begin{remark}
We thus think that, inside this flow, a thin boundary layer of thinness $%
mh_{i}$ occurs in the $i$-th direction with a probability $\overline{\lambda 
}_{i}$ (for every $i$).
\end{remark}

\noindent \textbf{Acknowledgement}.\ This work has been supported by the
Comit\'{e} Mixte Franco-Marocain under the Action Int\'{e}gr\'{e}e MA/04/93.

\end{document}